\documentclass[12pt]{amsart}
\usepackage{hyperref}
\usepackage{ amssymb }
\numberwithin{equation}{section}
\usepackage{hyperref}
\usepackage{float}
\usepackage{pgf,tikz}

\newcommand \reg{\operatorname{reg}}

\newcommand\link{\operatorname{link}}

\newcommand \G{\mathcal{G}}

\newcommand \B{\mathcal{B}}

\newcommand \K{\mathbb{K}}

\newcommand{\ass}{\operatorname{Ass}}

\newtheorem{theorem}{Theorem}[section]
\newtheorem{definition}[theorem]{Definition}
\newtheorem{cons}[theorem]{Construction}
\newtheorem{lemma}[theorem]{Lemma}

\newtheorem{example}[theorem]{Example}
\newtheorem{obs}[theorem]{Observation}
\newtheorem{question}[theorem]{Question}
\newtheorem{remark}[theorem]{Remark}

\newtheorem{corollary}[theorem]{Corollary}

\newtheorem*{notation*}{Notation}

\begin{document}
% \title[Symbolic powers of cover ideals]{Symbolic powers of cover ideals of vertex decomposable graphs}
\title[Componentwise linearity of powers of cover ideals]
{Componentwise linearity of powers of cover ideals}
%\title[Symbolic powers of vertex cover ideals]{Symbolic powers of vertex cover ideals}

\dedicatory{Dedicated to Professor J\"{u}rgen Herzog on the occasion of his 80th birthday}
  
\author{S. Selvaraja}
\email{selva.y2s@gmail.com}
\address{Chennai Mathematical Institute, H1, SIPCOT IT Park, Siruseri, Kelambakkam, Chennai,
INDIA - 603103.}
\author{Joseph W. Skelton}
\email{jskelton@tulane.edu}
\address{Tulane University, Department of Mathematics, 6823 St. Charles Ave., New Orleans,
LA 70118, USA}

\thanks{AMS Classification 2010: 13D02, 13F20, 05C25, 05E40} 
\maketitle

\begin{abstract}

Let $G$ be a finite simple graph and $J(G)$ denote 
its vertex cover ideal in a polynomial ring over a field. % $\mathbb{K}$.
The $k$-th symbolic power of $J(G)$ is denoted by $J(G)^{(k)}$. 
In this paper, we give a criteria for cover ideals of vertex decomposable graphs to 
have the property that all their symbolic powers are not componentwise linear. Also,
we give a necessary and sufficient condition on $G$ so that $J(G)^{(k)}$ is a
componentwise linear ideal for some (equivalently, for all) $k \geq 2$ when
$G$ is a graph such that $G \setminus N_G[A]$ has a simplicial vertex for any
independent set $A$ of $G$. Using this result,
we prove that  $J(G)^{(k)}$ is a
componentwise linear ideal for several
classes of graphs for all  $k \geq 2$. In particular,
if $G$ is a bipartite graph, then 
$J(G)$ is a componentwise linear ideal if and only if $J(G)^k$ is a componentwise linear
ideal for some (equivalently, for all) $k \geq 2$.
 \end{abstract}

\section{Introduction}
Let $R=\mathbb{K}[x_1,\ldots,x_n]$ be the polynomial ring in $n$ variables 
over a field $\mathbb{K}$.
In \cite{HerHibi}, Herzog 
 and Hibi introduced the concept of componentwise linear ideals.
A homogeneous ideal $I$ is called {\it componentwise linear} 
 if for each $\ell$, the ideal generated
 by all degree $\ell$ elements of $I$ has a linear resolution.
 It can be easily seen that ideals with linear resolution are componentwise linear.
 Componentwise linear ideals behave with respect to several properties 
 very much like ideals with a linear resolution (see, for example \cite[Section 8.2]{Herzog'sBook}).
 %Componentwise linear ideals have a number of algebraic and combinatorial 
 %properties that make them interesting to study, see \cite{Herzog'sBook}.
 %Ideals with linear quotients were defined by Herzog and Takayama \cite{HerTak}
 %in connection to their work on minimal free resolution of monomial ideals.
 A monomial ideal $I \subset R$ is said to have \textit{linear quotients}
 if there is an ordering $u_1 < \cdots < u_m$ on the minimal monomial generators of $I$ 
 such that for every $2 \leq i \leq m $ the ideal 
 $((u_1,\ldots,u_{i-1}) : (u_i))$ is generated by a subset of $\{x_1,\ldots, x_n\}$.
 This notion was introduced by  Herzog and Takayama in \cite{HerTak}.  
 %If a monomial ideal has linear quotients, then it has componentwise linear quotients \cite{JahanZheng}, 
 %and hence it is componentwise linear.

Ordinary and symbolic powers of ideals have been extensively studied for over two decades.
We refer the reader to \cite{EHHV20, DDAGHN} for a 
review of results in the literature.
%Let $R=\mathbb{K}[x_1,\ldots,x_n]$ be the polynomial ring in $n$ variables 
%over a field $\mathbb{K}$. 
For any arbitrary ideal $I \subseteq R$, the \textit{$k$-th symbolic power} of 
$I$ is the ideal
\[
 I^{(k)}=\bigcap\limits_{\mathfrak{p}\in \ass(I)}(I^kR_{\mathfrak{p}} \cap R),
\]
where $\ass(I)$ is the set of associated primes of $I$ and $R_{\mathfrak{p}}$ 
is the ring $R$ localized at the prime ideal $\mathfrak{p}$.
In the case that $I$ is a squarefree monomial ideal with primary decomposition
$I=\mathfrak{p}_1 \cap \cdots \cap \mathfrak{p}_r$, its $k$-th
symbolic power is given by $I^{(k)}=\mathfrak{p}_1^k \cap \cdots \cap \mathfrak{p}_r^k$
(\cite[Proposition 1.4.4]{Herzog'sBook}).
In this paper, we study  the symbolic powers of unmixed squarefree monomial ideals of height
two. These ideals are naturally associated to finite simple graphs and are 
called \textit{cover ideals}.
Ordinary and symbolic powers of edge/cover ideals of graphs have been studied by many 
authors (cf. \cite{EHHV20,DHNT20,Nursel,jayanthan,JS21,KK20,Fatemesh11,Fatemesh14,selva1,Fakhari}).

 Let $G=(V(G),E(G))$ denote a finite, simple (no loops, no multiple edges), undirected graph with 
 vertex set $V(G) =\{x_1,\ldots,x_n\}$ and edge set $E(G)$.
 A { \it minimal vertex cover} of $G$ is a subset $V \subseteq V(G)$ such that each 
 edge has at least one 
 vertex in $V$ and no proper subset of $V$ has the same property. 
 For a graph $G$, by identifying the vertices with variables in 
$R=\mathbb{K}[x_1,\ldots,x_n]$,  two squarefree monomial ideals are associated to $G$:
the edge ideal $I(G)$ generated by all monomials $x_ix_j$ with 
$\{x_i,x_j\} \in E(G)$, and the vertex cover ideal $J(G)$ 
generated by monomials $\prod\limits_{x_i \in V} x_i$ for all 
minimal vertex covers $V$ of $G$. This identification gives a one-to-one correspondence between a graph and its edge and cover ideals. 
The vertex cover ideal of a graph  $G$ is the Alexander dual of its edge ideal, i.e.,
$J(G)=I(G)^\vee=\bigcap\limits_{\{x_i, x_j\}\in E(G)} (x_i,x_j).$
The study of edge ideals and cover ideals from both the algebraic and 
combinatorial points of view has become a very active area of research in 
commutative algebra
(cf. \cite{BFH15, EHHV20, crupi, DHNT20, eagon, FH, fv, GRV05, HerHibi, Herzog'sBook, HHM20,
HerHibiOhsugi, HerReiWelker, HerTak, JahanZheng, JS21, Fatemesh11, Fatemesh14, Wood2}).

A graph $G$ is said to be \textit{vertex decomposable}({\it shellable})
if its independence complex  $\Delta(G)$ has  this  property (see Section \ref{pre} for the 
definition). 
A  graph $G$ is  called  \textit{sequentially  Cohen-Macaulay},  
if the quotient ring $R/I(G)$ is sequentially  Cohen-Macaulay.  
We have the chain of implications:
\begin{align*}
 \text{vertex-decomposable} \Longrightarrow \text{shellable}
 \Longrightarrow \text{sequentially
 Cohen-Macaulay.}
\end{align*}  
In \cite{eagon}, Eagon and Reiner's work with squarefree monomial ideals can be applied to these graph ideals giving us that $G$ is Cohen-Macaulay if and only if 
$I(G)^\vee = J(G)$ has a linear resolution. In \cite{HerHibi,HerReiWelker}, the authors
 proved that 
 $G$ is sequentially Cohen-Macaulay  if and only if 
 $J(G)$  is componentwise linear.  
 Francisco and Van Tuyl  proved that 
 if $G$ is a chordal graph, then $G$ is sequentially Cohen-Macaulay
 and hence $J(G)$ is componentwise linear \cite{fv}. %Now one can ask the following question:
%Given a (sequentially) Cohen-Macaulay graph, what can be said about the powers of its vertex cover ideal? 
   In \cite{HerHibiOhsugi}, Herzog, Hibi, and Ohsugi gave a condition on homogeneous ideals 
   having the property  that all their ordinary powers are componentwise linear. 
   They also conjectured that all powers of the vertex cover ideals of
   chordal graphs are componentwise linear.  
 There has been very little progress made on this conjecture except for very few classes 
 like generalized star graphs, tree, star graph based on a complete graphs,
 Cohen-Macaulay chordal graphs, bi-clique graphs, Cameron-Walker graphs whose bipartite graph is a
 complete bipartite graph \cite{HHM20,  HerHibiOhsugi, KK20, Fatemesh11}.
 The reader can refer to \cite{Nursel, HV22, Fatemesh14}
 for some recent articles where powers of componentwise linear cover ideals were studied.

 Symbolic powers of componentwise linear cover ideals are relatively less 
 explored than powers of componentwise linear cover ideals in the literature.
 For example, Fakhari \cite{Fakhari} proved that  if $G$ is a 
   Cohen-Macaulay and very well-covered graph,
   then  $J(G)^{(k)}$ is componentwise linear for all $k \geq 2$. Also, he 
   characterized all graphs $G$ with the property that $J(G)^{(k)}$ has a linear 
   resolution for some (equivalently, for all) $k \geq 2$ \cite{Se20}.
   In \cite{DHNT20}, Dung, Hien, Nguyen and Trung proved that all symbolic powers of the cover ideal 
   of $G \cup W(V(G))$, the graph obtained from $G$ by adding a whisker to each vertex in $V(G)$,  
   are componentwise linear.
   Then the first author of this paper proved that one only needs to 
   whisker at the vertices of a vertex cover of $G$ to get a new graph for which 
   all symbolic powers of the cover ideal are componentwise linear \cite{selva1}. 
   Thereafter, the second author along with Gu and  H\`a generalized this result, proving that the 
   same conclusion holds true after adding whiskers at the vertices of a cycle cover 
   of the graph \cite{GTS20}.

   In general, if $G$ is a sequentially Cohen-Macaulay graph, then the symbolic powers 
   of $J(G)$ need not be componentwise linear (\cite[Example 4.4]{selva1}). On the 
   other hand, we prove that if $J(G)^{(k)}$ and $J(G)^{(k+1)}$ are not componentwise 
   linear ideals for some $k \geq 1$, then $J(G)^{(k')}$ is not componentwise linear for all $k' \geq k$ (Theorem \ref{not-cl}). Therefore, one may ask:
   \begin{question}\label{question}
    Find a necessary and sufficient condition on $G$ such that $J(G)^{(k)}$ is a 
    componentwise linear ideal for all $k \geq 2$ when 
    $G$ is sequentially Cohen-Macaulay graph.
   \end{question}
   This paper revolves around the above question.
   We now restrict our attention to cover ideals of vertex decomposable graphs.   
   Let $G$ be a vertex decomposable graph. 
   To each vertex decomposable graph $G$, we associate the spanning bipartite subgraph 
   $\B_G$ of $G$ with partitions $\{x_{\gamma(1)},\ldots,x_{\gamma(r)}\} \sqcup 
   \{x_{\alpha(1)},\ldots,x_{\alpha(l)}\}$ and 
   $E(\B_G)=\{\{x_{\gamma(i)},x_{\alpha(j)}\} \mid 1 \leq i \leq r,~1 \leq j \leq l\}$,
   where $(x_{\alpha(1)},\ldots,x_{\alpha(l)})$ is a shedding order of $G$ and 
   $x_{\gamma(1)},\ldots,x_{\gamma(r)}$ are isolated vertices of 
   $G \setminus \{x_{\alpha(1)},\ldots,x_{\alpha(l)}\}$
   (Construction \ref{intr-cons}). 
   %The first main result of the paper partially answers Question \ref{question}.
   First, we give a necessary condition on $\B_G$ for $J(G)^{(k)}$ to be componentwise linear
   for all $k \geq 2$. More precisely, if there exists an independent 
   set $A$ of $G$ such that
 $\B_{G \setminus N_G[A]}$ is not vertex decomposable, then 
$J(G)^{(k)}$ is not a componentwise linear ideal for all $k \geq 2$ (Theorem \ref{suff-cond1}).
 Our second main result addresses the converse of above result for 
 certain class of vertex decomposable graphs. We  consider the 
 class of graphs $G$ such that 
 $G \setminus N_G[A]$ has a simplicial vertex for 
 any independent  set $A$ of $G$.
 In \cite{Russ11}, Woodroofe proved that if
 $G$ is a graph such that $G \setminus N_G[A]$ has a simplicial vertex for 
 any independent
 set $A$ of $G$, then $G$ is a vertex decomposable graph.
% We call a graph $G$ is a \textit{$W$-graph} 
%if $G$ satisfies the above hypothesis.
 This family of  graphs is rich, since it includes all
 chordal graphs, simplicial graphs, vertex decomposable bipartite graphs, etc.
 (see Remark \ref{wgraphs} for more classes of graphs).
 We prove the following result.
 \begin{theorem} (Theorem \ref{main})
  Let $G$ be a graph such that $G \setminus N_G[A]$ has a simplicial vertex for any 
  independent set $A$ of $G$. Then the following are
equivalent:
\begin{enumerate}
 \item $\B_{G\setminus N_G[A]}$ is a vertex decomposable graph
 for any independent set $A$ of $G$;
 \item $J(G)^{(k)}$ has linear quotients for all $k \geq 1$;
 \item $J(G)^{(k)}$ is a 
 componentwise linear ideal for all $k \geq 1$;
 \item $J(G)^{(k)}$ is a 
 componentwise linear ideal for some $k \geq 2$.
\end{enumerate}
 \end{theorem}

It may be noted that, as of now, there are no combinatorial characterizations for the
class of vertex decomposable bipartite graphs such as unmixed bipartite graphs or  
very well covered graphs.
However, Van Tuyl proved that the sequential Cohen-Macaulayness of a bipartite graph $G$ is equivalent to the
vertex decomposability of $G$ \cite[Corollary 2.12]{adam}. Also, Van Tuyl and Villarreal give a recursive 
characterization for a sequentially Cohen-Macaulay bipartite graph \cite{VanVilla}.
Using these results, one can verify whether a bipartite graph is vertex decomposable. 

% as follows:
%If $G$ is bipartite, then $G$ is sequentially Cohen-Macaulay if and only if 
 %there are adjacent vertices $x$ and $y$ with $\deg_G(x) = 1$ such that the bipartite 
 %graphs $G'= G \setminus N_G[x]$ and $G''= G \setminus N_G[y]$ are 
 %sequentially Cohen-Macaulay, 

   In \cite{GTS20},the authors raised the following question: find a necessary and sufficient condition on a subset $S$ of the vertices in a graph 
   $G$ such that $J(G \cup W(S))^{(k)}$ is componentwise linear  where 
   $G \cup W(S)$ is the graph obtained from $G$ by adding a whisker to each vertex in $S$
   for all $k \geq 1$. As a consequence of our investigation, we partially answers the above
   question (Corollary \ref{whisker-cor}, Corollary \ref{main-cor}, Corollary \ref{3-ver}).
   As an immediate consequence we  
 recover and extend the main results of
 \cite{DHNT20,GTS20, KK20, selva1}.
 Dung, Hien, Nguyen, and Trung asked the following question:
 classify all star graphs based on a complete graph $G$ such that all the symbolic powers of $J(G)$
 are componentwise linear \cite[Question 5.13]{DHNT20}.
 We answer this question affirmatively in Corollary \ref{star-main}.
 Herzog, Hibi and Moradi proved that if $G$ is a bi-clique graph 
 (see Definition \ref{n-clique-def}), then 
 $J(G)^k$ is componentwise linear for all $k \geq 2$ \cite{HHM20}.
 In this context it is natural to ask what happens when we consider the 
 symbolic powers of
cover ideals of $n$-clique graphs? 
As a consequence of our investigation, we obtain a necessary and sufficient condition
on $G$ so that  $J(G)^{(k)}$ is a
componentwise linear ideal for  all $k \geq 2$ (Corollary \ref{n-clique}).
 
 We also study the powers of cover ideals of bipartite graphs.
 In \cite[Theorem 2.2]{FS10}, Mohammadi and Moradi 
proved that if $G$ is a Cohen-Macaulay bipartite graph, then 
$J(G)^k$ has linear resolution for all $k \geq 1$.
%Fakhari proved that, let $G$ be a bipartite graph. Then 
Fakhari proved, for a bipartite graph $G$, the cover ideal 
$J(G)$ has linear resolution if and only if $J(G)^k$ has linear resolution for all
$k \geq 2$ (\cite[Corollary 3.7]{Fakhari}).  
 %Note that if $G$ is a vertex decomposable bipartite graph, then $G$ is a $W$-graph.
 By \cite[Corollary 2.6]{GRV05}, for every bipartite graph $G$, we have 
 $J(G)^k=J(G)^{(k)}$ for all $k \geq 2$.
 As an application of Theorem \ref{main}, we prove that if $G$ is a 
 bipartite graph, then $J(G)$ is a componentwise linear ideal if and only if 
 $J(G)^k$ is a componentwise linear ideal for some (equivalently, for all) $k \geq 2$,
 (Theorem \ref{main-bipartite}).
 Furthermore, we recover Kumar and Kumar's result (\cite[Corollary 3.4]{KK20}) on trees in Corollary \ref{tree-kk}.

 Our paper is organized as follows. In Section \ref{pre}, we collect the terminology and 
 preliminary results that are essential for the rest of the paper. 
 In Section \ref{new}, we give a criteria for cover ideals to have the property that
all their symbolic powers are not componentwise linear.
 In Section \ref{symbolic}, we prove that $J(G)^{(k)}$ is a componentwise linear
 ideal for all $k \geq 2$ when $G$ is a $W$-graph. We study the case when $G$ is a
 bipartite graph in Section \ref{bipartite} and show that in this case also,
 $J(G)^{k}$ is a componentwise linear
 ideal for all $k \geq 2$ when $J(G)$ is componentwise linear.

\section{Preliminaries}\label{pre}

In this section, we set up the basic definitions and notation needed for the main results. 
For a graph $G$, let $V(G)$ and $E(G)$ denote the set of all
vertices and the set of all edges of $G$, respectively.
For $\{x_1,\ldots,x_r\}  \subseteq V(G)$, let $N_G(x_1,\ldots,x_r) = 
\{y \in V (G)\mid \{x_i, y\} \in E(G)~ 
\text{for some $1 \leq i \leq r$}\}$ be the set of neighbors of $x_1,\ldots,x_r$
and $N_G[x_1,\ldots,x_r]= N_G(x_1,\ldots,x_r) \cup \{x_1,\ldots,x_r\}$. 
%and set $N_G[u]=N_G(u) \cup \{u\}$.
The cardinality of $N_G(x)$ is called the \textit{degree} of $x$ in $G$ and is denoted by
$\deg_G(x)$. 
A \textit{spanning subgraph} is a subgraph that contains all the vertices of the original graph.
A subgraph $H \subseteq G$  is called {\it induced} if for $u, v
\in V(H)$, $\{u,v\} \in
E(H)$ if and only if $\{u,v\} \in E(G)$.
A subset
$X$ of $V(G)$ is called {\it independent} if there is no edge $\{x,y\} \in E(G)$ 
for $x, y \in X$. 
A graph $G$ is called {\it bipartite} if there are two disjoint independent subsets $X$, $Y$ of
$V(G)$ such that $V(G)=X \sqcup Y$.
A {\it complete graph} is a graph in which each pair of vertices is connected by an edge. 
 %The complete graph with $n$ vertices is denoted $K_n$.
 A subset $U$ of $V(G)$ is said to be a \textit{clique} if the induced subgraph with vertex set $U$
is a complete graph.  
%A vertex $v$ is said to be \textit{free vertex} if it belongs to exactly one
%maximal clique. 
A \textit{simplicial vertex} of a graph $G$ is a vertex $x$ such that the neighbors of $x$ form a clique in
$G$. %The vertex $x$ is a {\it simplicial vertex} of $G$ if $N_G[x]$ is a complete graph.  
Note that if  $\deg_G(x)=1$, 
then $x$ is a simplicial vertex of $G$.  
 For $S \subseteq V(G)$, let $G \cup W(S)$ denote the graph on the 
vertex set $V(G) \cup \{z_x \mid x \in S\}$ whose edge set is
$E(G \cup W(S))=E(G) \cup \Big\{ \{x,z_x\} \mid x \in S\Big\}$.
A graph $G$ is \textit{chordal} if every induced cycle in $G$ has length 3.

We recall the relevant background on simplicial complexes.  
 A \textit{simplicial complex} $\Delta$ on $V = \{x_1,\ldots,x_n\}$ is
a collection of subsets of $V$ such that: 
\begin{enumerate}
 \item[(i)] $\{x_i\}\in \Delta $ for $1 \leq i\leq n$, and
 \item[(ii)] if $F \in \Delta$ and $F' \subseteq F$, then $F' \in \Delta$.
\end{enumerate}
Elements of $\Delta$ are called the \textit{faces} of $\Delta$, and the maximal elements, with 
respect to inclusion, are called the \textit{facets}. 
%A simplicial complex is \textit{pure} if all its facets have the same cardinality.
The link of a face $F$ in
$\Delta$ is $$\link_\Delta(F) = \{F' \mid F' \cup F \text{ is a face in
} \Delta, ~F' \cap F = \emptyset \}.$$
The \textit{deletion} of a face $F$ in 
$\Delta$ is $\Delta \setminus F=\{H \in \Delta \mid H \cap F =\emptyset\}$.

A simplicial complex $\Delta$ is
recursively defined to be {\em vertex decomposable} if it is either a
simplex or else has some vertex $v$ so that 
\begin{enumerate}
  \item[(i)] both $\Delta \setminus v$ and $\link_\Delta (v)$ are vertex decomposable, and
  \item[(ii)] no face of $\link_\Delta (v)$ is a facet of $\Delta \setminus v$.
\end{enumerate}

A simplicial complex $\Delta$ is \textit{shellable} if the facets of 
$\Delta$ can be ordered $F_1,\ldots,F_s$ such that for all $1 \leq i< j \leq s$, there 
exists some $v \in F_j \setminus F_i$ and some $\ell \in \{1,\ldots,j-1\}$ with
$F_j \setminus F_\ell=\{v\}$.

The \textit{independence complex} of $G$, denoted by $\Delta(G)$, is the simplicial
complex on $V(G)$ with face set 
$\Delta(G)=\{F \subseteq V(G) \mid F \text{ is an independent set of $G$ } \}.$
A graph $G$ is said to be \textit{vertex decomposable (shellable)} if $\Delta(G)$ is a
vertex decomposable (shellable) simplicial complex. 
In \cite{Wood2}, Woodroofe translated the definition of vertex decomposable
to graphs.
\begin{definition}\label{def-vertdecom} \cite[Lemma 4]{Wood2}
 A graph $G$ is recursively defined to be vertex decomposable if 
 $G$ is totally disconnected (with no edges) or if 
 \begin{enumerate}
  \item there is a vertex $x$ in $G$ such that $G \setminus x$ and 
  $G \setminus N_G[x]$ are both vertex decomposable, and
  \item no independent set in $G\setminus N_G[x]$ is a maximal independent set 
  in $G\setminus x$.
 \end{enumerate}
 \end{definition}
A \textit{shedding vertex} of $x$ is any vertex which satisfies 
 Condition (2) of Definition \ref{def-vertdecom}. 
 A graded $R=\K[x_1,\ldots,x_n]$-module $M$ is called \textit{sequentially Cohen-Macaulay} %(over $\K$) 
if there exists a finite filtration of graded $R$-modules 
$0=M_0 \subset M_1 \subset \cdots \subset M_{r}=M$ 
such that each quotient $M_i/M_{i-1}$ is Cohen-Macaulay, and the Krull dimensions of the 
quotients satisfy:
$\dim(M_1/M_0) < \dim(M_2/M_1) < \cdots < \dim(M_r/M_{r-1}).$
A graph $G$ is a \textit{sequentially Cohen-Macaulay graph
(Cohen-Macaulay)}  if $R/I(G)$ is sequentially Cohen-Macaulay
(Cohen-Macaulay).
\iffalse
 Let $G$ and $H$ be graphs. 
If $G$ and $H$ disjoint graphs (i.e., $V (G) \cap V (H) = \emptyset)$, we denote the disjoint union of 
$G$ and $H$ by $G\coprod H$.

\begin{corollary}\cite[Corollary 7(1) and Lemma 20]{Wood2}\label{cor-shedding} \mbox{}
\begin{enumerate}
 \item Any neighbor of a simplicial vertex of $G$ is a shedding vertex of $G$.
 \item Let $G$ and $H$ be two graphs.
Then $G \coprod H$ is vertex decomposable if and only if $G$ and $H$ are vertex decomposable.
\end{enumerate}
\end{corollary}
\fi

The property of 
sequentially Cohen-Macaulayness/shellable/
vertex decomposability is preserved when forming a link.
The following result is used repeatedly throughout this paper.

\begin{theorem}\cite[Theorem 2.5]{BFH15}, \cite[Theorem 2.6 and Theorem 3.3]{VanVilla}
\label{ind-sub}
Let $G$ be a sequentially Cohen-Macaulay/shellable/vertex decomposable graph.
Then $G\setminus N_G[A]$ is a sequentially 
Cohen-Macaulay/shellable/vertex decomposable graph for any independent 
set $A$ of $G$. This includes the case $A=\{x\}$ for any 
$x \in V(G)$.
 %If $G$ is sequentially Cohen-Macaulay graph/shellable/vertex decomposable, then 
 %$G \setminus N_G[x]$ is sequentially Cohen-Macaulay graph/shellable/vertex decomposable
 %for any $x \in V(G)$. In other words, 
 %if $G$ is sequentially Cohen-Macaulay graph/shellable/vertex decomposable, then 
 %$G \setminus N_G[A]$ is sequentially Cohen-Macaulay graph/shellable/vertex decomposable
 %for any independent set $A$ of $G$.
\end{theorem}

The following result shows that sequentially Cohen-Macaulay is a combinatorial property
for bipartite graphs.

\begin{theorem}\cite[Theorem 2.10]{adam}\label{adam-rs} Let $G$ be a bipartite graph. 
Then the following are equivalent: 
\begin{enumerate}
 \item $G$ is sequentially Cohen-Macaulay;
 \item $G$ is shellable;
 \item $G$ is vertex decomposable.
\end{enumerate}
\end{theorem}

In \cite{VanVilla}, Van Tuyl and Villarreal gave a recursive characterization 
for a sequentially Cohen-Macaulay bipartite graph as follows.

\begin{theorem}\cite[Corollary 3.11]{VanVilla}\label{adam-vila-rs}
 Let $G$ be a bipartite graph. Then $G$ is sequentially Cohen-Macaulay if and only if 
 there are adjacent vertices $x$ and $y$ with $\deg_G(x) = 1$ such that the bipartite 
 graphs $G'= G \setminus N_G[x]$ and $G''= G \setminus N_G[y]$ are 
 sequentially Cohen-Macaulay.
\end{theorem}

Linear quotients can  be used to verify that an ideal is componentwise linear:
\begin{theorem}\cite[Theorem 8.2.15]{Herzog'sBook}\label{lq-rs}
 Let $I$ be a homogeneous ideal, and suppose that $I$ has
linear quotients with respect to a minimal set of generators of $I$. Then $I$ is 
componentwise
linear.
\end{theorem}

In the study of symbolic powers of vertex cover ideals, 
Seyed Fakhari constructed a new graph $G_k$ 
whose vertex cover ideal is strongly related to the 
$k$-th symbolic power of the vertex cover ideal of $G$ \cite{Fakhari}.
We will make use of this construction extensively.
%This construction  has proved to be quite powerful, and we shall make use of it often.
 
 \begin{cons}\label{construction}
 Let $G$ be a graph with vertex set $V(G)= \{x_1,\ldots,x_n\}$ and $k \geq 1$
be an integer. We define the graph $G_k$ whose vertex set is
\[
 V(G_k)=\{x_{i,p} \mid 1 \leq i \leq n \text{ and } 1 \leq p \leq k\},
\]
and the edge set of $G_k$ is
\[
 E(G_k)=\Big\{\{x_{i,p},x_{j,q}\} \mid \{x_i,x_j\} \in E(G) \text{ and } p+q \leq k+1\Big\}.
\]
%Note that $(G \coprod H)_k=(G_k \coprod H_k)$, for any $k \geq 1$.
\end{cons}

We illustrate Construction \ref{construction} with the help of the following example.
\begin{example}
 Let $G=(V(G),E(G))$ be a graph with $V(G)=\{x_1,x_2\}$ and $E(G)=\{\{x_1,x_2\}\}$. 
 Then $G_2$ is the graph with vertex set $\{x_{1,1}, x_{1,2}, x_{2,1}, x_{2,2}\}$ and the 
 edge set of $G_2$ is
 \[
  E(G_2)=\{\{x_{1,p}, x_{2,q}\} \mid p+q \leq 3\}=\{\{x_{1,1},x_{2,1}\},
  \{x_{1,1},x_{2,2}\},\{x_{2,1}, x_{1,2}\}\}.
 \]

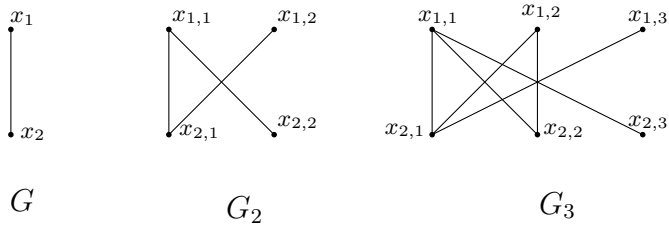
 \begin{figure}[H]
  \begin{tikzpicture}[scale=0.7]
%\clip(-0.2,0.23) rectangle (19.07,10.18);
\draw (2,5)-- (2,3);
\draw (5,5)-- (5,3);
\draw (7,3)-- (5,5);
\draw (7,5)-- (5,3);
\draw (10,5)-- (10,3);
\draw (12,3)-- (10,5);
\draw (12,5)-- (10,3);
\draw (14,5)-- (10,3);
\draw (10,5)-- (14,3);
\draw (12,5)-- (12,3);
\draw (1.77,2.17) node[anchor=north west] {$ G $};
\draw (5.88,2.04) node[anchor=north west] {$G_2$};
\draw (11.82,2.12) node[anchor=north west] {$G_3$};
\begin{scriptsize}
\fill [color=black] (2,5) circle (1.5pt);
\draw[color=black] (2.22,5.21) node {$x_1$};
\fill [color=black] (2,3) circle (1.5pt);
\draw[color=black] (2.42,3.01) node {$x_2$};
\fill [color=black] (5,5) circle (1.5pt);
\draw[color=black] (5.45,5.21) node {$x_{1,1}$};
\fill [color=black] (5,3) circle (1.5pt);
\draw[color=black] (5.58,3.03) node {$x_{2,1}$};
\fill [color=black] (7,3) circle (1.5pt);
\draw[color=black] (7.45,3.21) node {$x_{2,2}$};
\fill [color=black] (7,5) circle (1.5pt);
\draw[color=black] (7.45,5.21) node {$x_{1,2}$};
\fill [color=black] (10,5) circle (1.5pt);
\draw[color=black] (10.12,5.21) node {$x_{1,1}$};
\fill [color=black] (10,3) circle (1.5pt);
\draw[color=black] (9.49,3.06) node {$x_{2,1}$};
\fill [color=black] (12,3) circle (1.5pt);
\draw[color=black] (12.5,3.05) node {$x_{2,2}$};
\fill [color=black] (12,5) circle (1.5pt);
\draw[color=black] (12.08,5.32) node {$x_{1,2}$};
\fill [color=black] (14,5) circle (1.5pt);
\draw[color=black] (14.11,5.21) node {$x_{1,3}$};
\fill [color=black] (14,3) circle (1.5pt);
\draw[color=black] (14.11,3.21) node {$x_{2,3}$};
\end{scriptsize}
\end{tikzpicture}
 \end{figure}
Also  $G_3$ can be thought of as duplicating the vertices of a graph 3 times and 
 connecting them under the constraints of the construction. 
\end{example}

 Throughout this paper, $G_k$ denotes the graph as in Construction \ref{construction}.
 
 Polarization is a process to obtain a squarefree monomial ideal from
a given monomial ideal. For details of polarization we refer to 
\cite[Section 1.6]{Herzog'sBook}.
\begin{definition}
 Let $f=x_1^{m_1} \cdots x_n^{m_n}$ be a monomial in $R=\K[x_1,\ldots,x_n]$.
 Let $\widetilde{R}=\K[x_{11},x_{12},\ldots,x_{21},x_{22},\ldots,x_{n1},x_{n_2},\ldots]$. 
 Then a \textit{polarization} of $f$ in $\widetilde{R}$ is the squarefree monomial 
 $\widetilde{f}= x_{11} \cdots x_{1m_1}x_{21}\cdots x_{2m_2} \cdots x_{n1} \cdots x_{nm_n}$. 
 If $f_1,\ldots,f_m \in R$ are monomials and 
 $I = (f_1,\ldots,f_m)$, then we call the squarefree monomial ideal $\widetilde{I}$ generated by the polarization
 of the $f_i$'s in a larger polynomial ring $\widetilde{R}$,
 the polarization of $I$.
\end{definition}
 %Polarization is a process that creates a squarefree monomial ideal (in
%a possibly different polynomial ring) from a given monomial ideal
%\cite[Section 1.6]{Herzog'sBook}. If $I$ is a monomial ideal in $R$, then the 
%polarization of
%$I$ is denoted by $\widetilde{I} \subseteq \widetilde{R}$.  
 The following lemma, due to Seyed Fakhari, is used repeatedly throughout this paper
\begin{lemma}\cite[Lemma 3.4]{Fakhari}\label{fak-result}
 Let $G$ be a graph. For every integer $k \geq 1$, the ideal $\widetilde{(J(G)^{(k)})}$ is the vertex cover ideal of $G_k$.
\end{lemma}

Let $M$ be a graded $R$ module. Recall that for non-negative
integers $i, j$ the symbol $\beta_{i,j}(M)$ denotes the $(i,j)$-th \emph{graded Betti
number} of $M$ \cite[Section A.2]{Herzog'sBook}. The 
\emph{Castelnuovo-Mumford regularity} of $M$, denoted by $\reg(M)$, is defined as 
$\reg(M)=\max \{j-i \mid \beta_{i,j}(M) \neq 0\}$. 

The following properties of the polarization process will be used in the next 
sections.
%In this paper, we repeatedly use the following properties of the
%polarization, namely:
\begin{lemma}\label{pol_reg}  Let $I$ be a monomial ideal in 
$R.$ Then
\begin{enumerate}
 \item \cite[Corollary 1.6.3]{Herzog'sBook} for 
 all $\ell,j$, $\beta_{\ell,j}(R/I)=\beta_{\ell,j}(\widetilde{R} / \widetilde {I})$.
 \item \cite[Lemma 3.5]{Fakhari} $I$ has linear quotients if and only if $\widetilde{I}$ has linear
quotients.
\item \cite[Proposition 1]{NPY21} $I$ is a componentwise linear ideal if and only if 
$\widetilde{I}$ is a componentwise linear ideal.
\end{enumerate}
\end{lemma}

\section{Symbolic powers of cover ideals of vertex decomposable graphs}\label{new}
In this section, we give a combinatorial sufficient condition for when 
$J(G)^{(k)}$ is not a componentwise linear ideal for all $k \geq 2$.
First, we  show that all the symbolic powers of $J(G)$ are not componentwise linear ideal,
provided $J(G)^{(k)}$ is not componentwise linear for particular values of $k$.

\begin{lemma}\label{n-lemma}
 Let $G$ be a graph.
 \begin{enumerate}
  \item If 
$J(G)^{(2k)}$ is not componentwise linear for some $k \geq 1$, then 
$J(G)^{(2k')}$ is not componentwise linear for all $k' > k$.
\item If $J(G)^{(2k+1)}$ is not componentwise linear for some $k \geq 0$, then 
$J(G)^{(2k'+1)}$ is not componentwise linear for all $k' > k$.
 \end{enumerate}
\end{lemma}
\begin{proof}
 Let $G$ be a graph with $V(G)=\{x_1,\ldots,x_n\}$ and $G_p$ be a graph as in 
 Construction \ref{construction} for all $p \geq 1$.
 \vskip 1mm
 \noindent
 (1) 
 Suppose $J(G)^{(2k)}$ is not  componentwise linear  for some $k \geq 1$.
 Then  $\widetilde{J(G)^{(2k)}}$ is not  componentwise linear.
 By Lemma \ref{fak-result}, $J(G_{2k})$ is not  componentwise linear and 
 hence $G_{2k}$ is not sequentially Cohen-Macaulay. 
 It follows from  \cite[Lemma 3.3]{selva1}
 that 
 \[
  G_p\setminus N_{G_p}[x_{1,p},\ldots,x_{n,p}] \simeq G_{p-2}~\text{ for all } p \geq 2.
 \]
 Suppose $G_{2k+2}$ is a sequentially Cohen-Macaulay graph.
 Since $x_{1,2k+2},\ldots,x_{n,2k+2}$ is an independent set of $G_{2k+2}$, 
 by Theorem \ref{ind-sub},
 $G_{2k+2}\setminus N_{G_{2k+2}}[x_{1,2k+2},\ldots,x_{n,2k+2}]$ is 
 sequentially Cohen-Macaulay. Note that
 $G_{2k+2}\setminus N_{G_{2k+2}}[x_{1,2k+2},\ldots,x_{n,2k+2}] \simeq G_{2k}$.
 This is a contradiction to 
 our assumption.
 Therefore 
 $G_{2k+2}$ is not sequentially Cohen-Macaulay and hence
 $J(G)^{(2k+2)}$ is not componentwise linear.
  Now proceeding as in the above, one can show that 
 $G_{2k'}$  is not a sequentially Cohen-Macaulay graph
 for all $k' \geq k$. 
 Therefore, $J(G)^{(2k')}$ is not a componentwise linear ideal for all  $k' \geq k$. 
 \vskip 1mm
 \noindent
 (2) If $J(G)^{(2k+1)}$ is not componentwise linear for some $k \geq 0$,
 then proceeding as in the previous case of the proof, one
can show that $J(G)^{(2k'+1)}$ is not componentwise linear for all $k' > k$.
\end{proof}

\begin{theorem}\label{not-cl}
 Let $G$ be a graph.  
 If $J(G)^{(k)}$ and $J(G)^{(k+1)}$ are not componentwise linear ideals for some $k \geq 1$, then 
 $J(G)^{(k')}$ is not componentwise linear for all $k' \geq k$.
 %If $J(G)^{(k)}$ is not  a componentwise linear ideal either $k=1,2$ or
 %$k=2,3$, then 
%$J(G)^{(k)}$ is not a componentwise linear ideal for all $k \geq 2$.
\end{theorem}
\begin{proof}
Suppose $J(G)^{(k)}$ and $J(G)^{(k+1)}$ are not componentwise linear ideals.
 It follows from Lemma \ref{n-lemma} that $J(G)^{(k')}$ is not
 componentwise linear for all $k' \geq k$.
\end{proof}

\begin{remark}
As an immediate consequence of Theorem \ref{not-cl}, one can see that if 
$J(G)$ and $J(G)^{(2)}$ are not componentwise linear ideals, then 
$J(G)^{(k)}$ is not componentwise linear for all $k \geq 3$.
If $J(G)^{(2)}$ and $J(G)^{(3)}$ are not componentwise linear ideals, then 
$J(G)^{(k)}$ is not componentwise linear for all $k \geq 4$.
\end{remark}

Let $G$ be a vertex decomposable graph. Recall that the 
sequence $(x_{\alpha(1)},\ldots,x_{\alpha(l)})$ is a 
\textit{shedding ordered} of $G$ if $x_{\alpha(1)}$ is a shedding vertex 
of $G$ and $x_{\alpha(i)}$ is a shedding vertex of $G \setminus 
\{x_{\alpha(1)},\ldots,x_{\alpha(i-1)}\}$ for all $1<i \leq l$. Given 
a shedding order $\{x_{\alpha(1)},\ldots,x_{\alpha(l)}\}$, we 
call the set of isolated vertices in 
$G \setminus \{x_{\alpha(1)},\ldots,x_{\alpha(l)}\}=
\{x_{\gamma(1)},\ldots,x_{\gamma(r)}\}$ an \textit{$i$-order} for $G$.
Note that an $i$-order of $G$ is an independent set and 
$V(G)$ is the disjoint union of the sets 
$\{x_{\alpha(1)},\ldots,x_{\alpha(l)}\}$ and 
$\{x_{\gamma(1)},\ldots,x_{\gamma(r)}\}$.
In the following construction, 
for every vertex decomposable graph $G$, we associate a bipartite spanning subgraph $\B_G$
which is strongly related to the symbolic powers of 
componentwise linear cover ideals.

\begin{cons} \label{intr-cons}
Let $G$ be a vertex decomposable graph.
Let $(x_{\alpha(1)},\ldots,x_{\alpha(l)})$ be a 
shedding order of $G$ and 
$\{x_{\gamma(1)},\ldots,x_{\gamma(r)}\}$
be the $i$-order of $G$. Then we associate to $G$ a 
spanning bipartite  subgraph
$\B_G$ with partitions
$$ \{x_{\gamma(1)},\ldots,x_{\gamma(r)}\} \sqcup
\{x_{\alpha(1)},\dots,x_{\alpha(l)}\}$$
 and 
 \[
  E(\B_G)=\Big\{\{x_{\gamma(i)},x_{\alpha(j)}\} \in E(G) \mid 1 \leq i \leq r,~1 \leq j \leq l 
  \Big\}.
 \]

\end{cons}

The next example describes the above construction.
 %For the convenience of the readers, we give an example 
 %in below, describing the procedure that we followed  above.

\begin{example}\label{ex}
 Let $G$ be a graph as shown in Fig. 1.
 Clearly, $G$ is a vertex decomposable graph.
 Note that $x_1$ is a simplicial vertex of $G$ and $N_{G}(x_1)=\{x_{2}\}$.
 
 \begin{minipage}{\linewidth}
  %\centering
\begin{minipage}{0.45\linewidth}
 
 \begin{figure}[H]
\begin{tikzpicture}[scale=0.3]
%\clip(-1.37,-1.54) rectangle (47.98,25.12);
\draw (10,16)-- (10,12);
\draw (10,12)-- (14,12);
\draw (10,12)-- (10,8);
\draw (14,12)-- (10,8);
\draw (10,4)-- (10,8);
\draw (10,8)-- (14,8);
\draw (14,8)-- (14,4);
\draw (10,4)-- (14,4);
\draw (20,16)-- (20,12);
\draw (20,12)-- (24,12);
\draw (20,8)-- (24,12);
\draw (20,8)-- (24,8);
\draw (20,4)-- (24,4);
\draw (20,8)-- (20,4);
\draw (24,8)-- (24,4);
\draw (7.66,3.04) node[anchor=north west] {Fig. 1. $G$};
\draw (19.57,3.17) node[anchor=north west] {Fig. 2. $\mathcal{B}_G$};
\begin{scriptsize}
\fill [color=black] (10,16) circle (3.5pt);
\draw[color=black] (10.33,16.55) node {$x_1$};
\fill [color=black] (10,12) circle (3.5pt);
\draw[color=black] (10.63,12.56) node {$x_2$};
\fill [color=black] (14,12) circle (3.5pt);
\draw[color=black] (14.32,12.56) node {$x_3$};
\fill [color=black] (10,8) circle (3.5pt);
\draw[color=black] (9.35,8.15) node {$x_4$};
\fill [color=black] (10,4) circle (3.5pt);
\draw[color=black] (9.43,4.16) node {$x_7$};
\fill [color=black] (14,8) circle (3.5pt);
\draw[color=black] (14.28,8.58) node {$x_5$};
\fill [color=black] (14,4) circle (3.5pt);
\draw[color=black] (14.62,4.15) node {$x_6$};
\fill [color=black] (20,16) circle (3.5pt);
\draw[color=black] (20.32,16.55) node {$x_1$};
\fill [color=black] (20,12) circle (3.5pt);
\draw[color=black] (20.6,12.56) node {$x_2$};
\fill [color=black] (24,12) circle (3.5pt);
\draw[color=black] (24.23,12.56) node {$x_3$};
\fill [color=black] (20,8) circle (3.5pt);
\draw[color=black] (19.38,8.19) node {$x_4$};
\fill [color=black] (24,8) circle (3.5pt);
\draw[color=black] (24.31,8.58) node {$x_5$};
\fill [color=black] (20,4) circle (3.5pt);
\draw[color=black] (19.38,4.2) node {$x_7$};
\fill [color=black] (24,4) circle (3.5pt);
\draw[color=black] (24.65,4.15) node {$x_6$};
\end{scriptsize}
\end{tikzpicture}
 \end{figure}
\end{minipage}
\begin{minipage}{0.5\linewidth}
 By \cite[Corollary 7]{Wood2} and \cite[Corollary 3.5]{selva1}, $x_2$ is a shedding vertex and $G \setminus x_2$
 is a vertex decomposable graph. Similarly, one can see that 
 $x_4$ is a shedding vertex of $G \setminus \{x_2\}$ and 
 $x_6$ is a shedding vertex of $G \setminus \{x_2,x_4\}$. Note that
 $G \setminus \{x_2,x_4\}$ is vertex decomposable and 
 $G \setminus \{x_2,x_4,x_6\}$ is totally disconnected. Therefore
 $(x_2,x_4,x_6)$ is a shedding order and $\{x_1,x_3,x_5,x_7\}$ is an $i$-order of 
 $G$.
\end{minipage}
\end{minipage} 
\end{example}

Now we recall the definition of a graph isomorphism from \cite[Definition 1.1.20]{west}.
 An isomorphism from a graph $G$ to a graph $H$ is a bijection 
 $\phi: V(G) \longrightarrow V(H)$ such that $\{u,v\} \in E(G)$ if and only if
 $\{\phi(u),\phi(v)\} \in E(H)$. We say $G$ is \textit{isomorphic} to $H$,
 written $G \simeq H$, if there is an isomorphism from $G$ to $H$.
Recall that the sequential Cohen-Macaulayness of a bipartite graph $\B_G$ is equivalent to the
graph $\B_G$  being vertex decomposable (Theorem \ref{adam-rs}).
Below we prove that if $\B_G$ is  not sequentially Cohen-Macaulay, then 
neither is $G_k$ for all $k \geq 2$. This result is crucial in obtaining our main result.

\begin{lemma}\label{suff-cond}
 Let $G$ be a vertex decomposable graph and $G_k$ be a graph as in Construction \ref{construction} for 
 all $k \geq 2$.
 If $\B_G$ is not vertex decomposable, then 
 $G_k$ is not a sequentially Cohen-Macaulay graph for all $k \geq 2$.
\end{lemma}
\begin{proof}
Let 
$(x_{\alpha(1)},\dots,x_{\alpha(l)})$ be a shedding order of $G$ and  
$\{x_{\gamma(1)},\ldots,x_{\gamma(r)}\}$ be 
an $i$-order of $G$. 
%Since $\{x_{\gamma(1)},\ldots,x_{\gamma(r)}\}$ is an independent set of% $G$, we have $N_G(\{x_{\gamma(1)},\ldots,x_{\gamma(r)}\})$ $=$ 
% $\{x_{\alpha(1)},\dots,x_{\alpha(l)}\}$.
 By the proof of Theorem \ref{not-cl}, it is enough to prove that $G_2$ and 
 $G_3$ are not sequentially Cohen-Macaulay
 graphs.
 Suppose  
 $G_2$ and $G_3$ are sequentially Cohen-Macaulay
 graphs.
 Note that $A=\{x_{\gamma(1),2},\ldots,x_{\gamma(r),2}\}$ 
and  
 $B=\{x_{\gamma(1),2},\ldots,x_{\gamma(r),2}$, 
 $x_{\gamma(1),3},\ldots,x_{\gamma(r),3}\}$ are
 independent sets of $G_2$ and $G_3$ respectively. 
 Therefore,  by 
 Theorem \ref{ind-sub}, $G_2 \setminus N_{G_2}[A]$
 and $G_3 \setminus N_{G_3}[B]$ are sequentially Cohen-Macaulay graphs. 
 Since $N_G(\{x_{\gamma(1)}$, $\ldots$, $x_{\gamma(r)}\})$ $=$ 
 $\{x_{\alpha(1)}$, $\dots$, $x_{\alpha(l)}\}$, we have 
 \begin{align*}
 N_{G_2}(A) &= 
 \{x_{\alpha(1),1}, \dots, x_{\alpha(l),1}\} \text{ and }\\
 N_{G_3}(B) &= 
 \{x_{\alpha(1),1}, \dots, x_{\alpha(l),1}, x_{\alpha(1),2}, \dots, 
 x_{\alpha(l),2}\}.
 \end{align*}

 It follows from Construction \ref{construction} that
$x_{\gamma(1),1},
 \ldots,x_{\gamma(r),1} 
 \notin N_{G_2}[A]$
  and $x_{\gamma(1),1},
 \ldots,x_{\gamma(r),1} 
 \notin N_{G_2}[B]$.
 Then 
 \begin{align*}
  V(G_2 \setminus N_{G_2}[A])&=
 \{x_{\gamma(1),1}, \ldots, x_{\gamma(r),1}, x_{\alpha(1),2}, \dots, 
 x_{\alpha(l),2} \}\\
 V(G_3 \setminus N_{G_3}[B])&=
 \{x_{\gamma(1),1}, \ldots, x_{\gamma(r),1}, x_{\alpha(1),3}, \dots, 
 x_{\alpha(l),3} \}.
 \end{align*}
 Clearly, $|V(G_2 \setminus N_{G_2}[A])|=|\B_G|$ and 
$|V(G_3 \setminus N_{G_3}[B])|=|\B_G|$. Since 
$\{x_{\gamma(1)},\ldots,x_{\gamma(r)}\}$ is an independent set of $G$,
$\{x_{\gamma(i),1}, x_{\gamma(j),1}\} \notin E(G_2 \setminus N_{G_2}[A])$,
 and similarly for $E(G_3\setminus N_{G_3}[B])$,
 %$\{x_{\gamma(i),1}, x_{\gamma(j),1}\} \notin E(G_3 \setminus N_{G_3}[B])$
 for all $1 \leq i,j \leq l$. With this, and Construction \ref{construction}, 
 it is clear
 \begin{align*}
  E(G_2 \setminus N_{G_2}[A])&= 
\{\{x_{\gamma(i),1}, x_{\alpha(j),2}\}\mid \{x_{\gamma(i)},x_{\alpha(j)}\}\in E(G_2)\} 
 \text{ and }\\
 E(G_3 \setminus N_{G_3}[B])&=\{\{x_{\gamma(i),1},
 x_{\alpha(j),3}\}\mid \{x_{\gamma(i)},x_{\alpha(j)}\}\in E(G_3)\}.
 \end{align*}
 %Let $\Phi:V(G_2 \setminus N_{G_2}[A]) \to V(\B_G)$ and $\Psi:V(G_3 \setminus 
 %N_{G_3}[B]) \to V(\B_G)$ be the maps defined by 
 %\begin{align*}
  %\Phi(x_{\gamma(i),1})&=x_{\gamma(i)},~\Phi(x_{\alpha(j),2})=x_{\alpha(j)} 
  %\text{ for all } 1\leq i \leq r,~1\leq j \leq l \text{ and }\\
  %\Psi(x_{\gamma(i),1})&=x_{\gamma(i)},~\Psi(x_{\alpha(j),3})=x_{\alpha(j)} 
  %\text{ for all } 1\leq i \leq r,~1\leq j \leq l.
 %\end{align*}
Then $G_2 \setminus N_{G_2}[A]$ is isomorphic to $\B_G$ and 
 $G_3 \setminus N_{G_3}[B]$ is isomorphic to $\B_G$.
 Since $\B_G$ is not sequentially Cohen-Macaulay, 
 $G_2 \setminus N_{G_2}[A]$ and $G_3 \setminus N_{G_3}[B]$ are not 
 sequentially Cohen-Macaulay graphs. This is a contradiction
 to our assumption. Hence $G_2$ and $G_3$ are not sequentially Cohen-Macaulay graphs.
\end{proof}

Now we are ready to prove the main result of this section.

\begin{theorem}\label{suff-cond1}
 Let $G$ be a vertex decomposable graph. If there exists an independent set $A$ of $G$ such that
 $\B_{G \setminus N_G[A]}$ is not vertex decomposable, then 
$J(G)^{(k)}$ is not a componentwise linear ideal for all $k \geq 2$.
 \end{theorem}
\begin{proof}
 If $A=\emptyset$, then by Lemma \ref{suff-cond} and Lemma \ref{fak-result}, $J(G)^{(k)}$ is not
 componentwise linear for all $k \geq 2$. Assume that $A \neq \emptyset$.
 Let $A=\{z_1,\ldots,z_s\}$ be an independent set of $G$. Assume that 
 $J(G)^{(k)}$ is componentwise linear for some $k \geq 2$. By Lemma \ref{fak-result},
 $G_k$ is a sequentially Cohen-Macaulay graph. It follows from 
 \cite[Lemma 3.3(3)]{selva1} that 
 \[
  G_k \setminus N_{G_k}[z_{1,1},\ldots,z_{s,1}]=(G \setminus N_G[z_1,\ldots,z_s])_k
  \cup \{\text{ isolated vertices }\}.
 \]
Since $A$ is an independent set of $G$, $\{z_{1,1},\ldots,z_{s,1}\}$ is an independent
set of $G_k$. By Lemma \ref{ind-sub}, $(G \setminus N_G[z_1,\ldots,z_s])_k$ is 
sequentially Cohen-Macaulay. Therefore, by Lemma \ref{suff-cond},
$\B_{G \setminus N_G[A]}$ is a vertex decomposable graph. This is a contradiction to 
our assumption. Hence $J(G)^{(k)}$ is not a componentwise linear ideal for all $k \geq 2$.
\end{proof}

The following example shows that, using Theorem \ref{suff-cond1}, we can easily construct a graph 
$G$ such that $J(G)$ is componentwise linear but $J(G)^{(k)}$ is not 
componentwise linear for all $k \geq 2$. A graph which is isomorphic to the graph with vertices $a$, $b$, $c$, $d$ and edges
$\{a,b\}$, $\{b,c\}$, $\{a,c\}$, $\{a,d\}$, $\{c,d\}$ is called a \textit{diamond}.

\begin{example}
Let $G$ be a vertex decomposable graph
and $G \setminus N_G[A]$ be isomorphic to the diamond graph 
 for 
some independent set $A$ of $G$. 
Then $J(G)$ is a componentwise linear ideal.

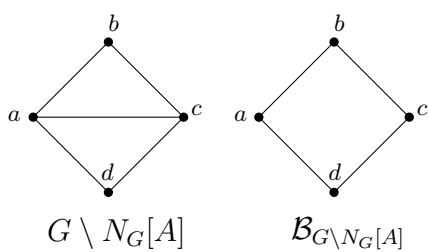
\begin{minipage}{\linewidth}
  %\centering
\begin{minipage}{0.36\linewidth}

\begin{figure}[H]
\begin{tikzpicture}[scale=0.5]
\draw (3,6)-- (1,4);
\draw (5,4)-- (3,6);
\draw (3,2)-- (1,4);
\draw (5,4)-- (3,2);
\draw (1,4)-- (5,4);
\draw (9,6)-- (7,4);
\draw (9,6)-- (11,4);
\draw (9,2)-- (7,4);
\draw (11,4)-- (9,2);
\draw (1.08,1.63) node[anchor=north west] {$ G\setminus N_G[A] $};
\draw (7.56,1.63) node[anchor=north west] {$ \mathcal{B}_{G\setminus N_G[A]} $};
\begin{scriptsize}
\fill [color=black] (3,6) circle (3.5pt);
\draw[color=black] (3.14,6.5) node {$b$};
\fill [color=black] (1,4) circle (3.5pt);
\draw[color=black] (0.5,4.04) node {$a$};
\fill [color=black] (5,4) circle (3.5pt);
\draw[color=black] (5.34,4.2) node {$c$};
\fill [color=black] (3,2) circle (3.5pt);
\draw[color=black] (2.98,2.55) node {$d$};
\fill [color=black] (9,6) circle (3.5pt);
\draw[color=black] (9.14,6.5) node {$b$};
\fill [color=black] (7,4) circle (3.5pt);
\draw[color=black] (6.52,4.03) node {$a$};
\fill [color=black] (11,4) circle (3.5pt);
\draw[color=black] (11.34,4.2) node {$c$};
\fill [color=black] (9,2) circle (3.5pt);
\draw[color=black] (9.02,2.55) node {$d$};
\end{scriptsize}
\end{tikzpicture}
\end{figure}
\end{minipage}
\begin{minipage}{0.59\linewidth}
Suppose $G\setminus N_G[A]$ is the graph 
with vertices $a,b,c,d$ and edges
$\{a,b\}$, $\{b,c\}$, $\{a,c\}$, $\{a,d\}$, $\{c,d\}$.
One can see that $\B_{G \setminus N_G[A]}$ is the graph
with vertices $a,b,c,d$ and edges $\{a,b\}$, $\{b,c\}$, $\{a,d\}$,
$\{c,d\}$. By Theorem \ref{adam-rs}, Theorem \ref{adam-vila-rs}, 
$\B_{G \setminus N_G[A]}$ is not vertex decomposable.
Therefore, by Theorem \ref{suff-cond1}, $J(G)^{(k)}$ is not componentwise linear
for all $k \geq 2$.
\end{minipage}
\end{minipage}
\end{example}

  It is known that if 
  $G \setminus S$ is not a sequentially Cohen-Macaulay graph, where $S \subseteq V(G)$, then 
  $G \cup W(S)$ is not sequentially Cohen-Macaulay (\cite[Theorem 4.1]{FH}).
  Inspired by this result, we show that if the spanning bipartite subgraph of 
  $G \setminus S$ is not 
  sequentially Cohen-Macaulay, then $J(G \cup W(S))^{(k)}$ is not componentwise 
  linear for all $k \geq 2$. First we observe that if 
  $G \cup W(S)$ is a vertex decomposable graph, then $G \setminus S$ is a vertex decomposable 
  graph.
  
 \begin{corollary}\label{whisker-cor}
  Let $G \cup W(S)$, $S \subseteq V(G)$, be a vertex decomposable graph. 
  If $\B_{(G\setminus S) \setminus N_{(G \setminus S)}[A]}$ is not vertex 
  decomposable for some 
  independent set $A$ of $G \setminus S$, then $J(G \cup W(S))^{(k)}$ is not componentwise linear for all 
  $k \geq 2$.
 \end{corollary}
\begin{proof}
Let $\G$ be a graph with vertex set $V(G) \cup \{z_x \mid x \in S\}$ and 
edge set $$E(G) \cup \{ \{x,z_x\} \mid x \in S\}.$$
One can see that $$(G\setminus S) \setminus N_{(G \setminus S)}[A]=
G \cup W(S) \setminus N_{G \cup W(S)}[A \cup \{z_x \mid x \in S\}]$$
for any independent set $A$ of $G \setminus S$.
Then $\B_{(G \setminus S) \setminus N_{G \setminus S}[A]}=
 \B_{G \cup W(S) \setminus N_{G \cup W(S)}[A \cup \{z_x \mid x \in S\}]}$
 for any independent set $A$ of $G \setminus S$.
Therefore, by Theorem \ref{suff-cond1}, $J(G \cup W(S))^{(k)}$ is not componentwise linear for all 
  $k \geq 2$.
 \end{proof}

 The following example shows that the converse of Corollary \ref{whisker-cor} 
 is not  necessarily true.
 \begin{example}\label{not-com}
  Let $G$ be a graph with $V(G)=\{x_1,\ldots,x_6\}$ and 
  set $\G=G \cup W(x_6)$.

 
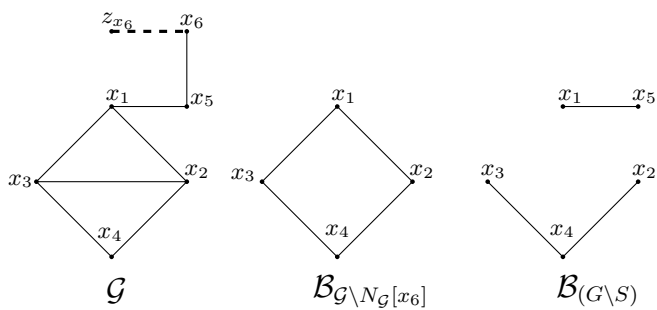
\begin{figure}[H]
\begin{tikzpicture}[scale=0.25]\draw (12,16)-- (8,12);
\draw (16,12)-- (12,16);
\draw (12,8)-- (8,12);
\draw (16,12)-- (12,8);
\draw (12,16)-- (16,16);
\draw (16,20)-- (16,16);
\draw [line width=1.2pt,dash pattern=on 3pt off 3pt] (12,20)-- (16,20);
\draw (24,16)-- (28,12);
\draw (24,16)-- (20,12);
\draw (24,8)-- (20,12);
\draw (28,12)-- (24,8);
\draw (36,8)-- (32,12);
\draw (40,12)-- (36,8);
\draw (36,16)-- (40,16);
\draw (11.12,7.46) node[anchor=north west] {$ \mathcal{G} $};
\draw (22.11,7.59) node[anchor=north west] {$ 
\mathcal{B}_{\mathcal{G} \setminus N_{\mathcal{G}}[x_6]} $};
\draw (35.07,7.63) node[anchor=north west] {$ \mathcal{B}_{(G\setminus S)} $};
\draw (8,12)-- (16,12);
\begin{scriptsize}
\fill [color=black] (12,16) circle (3.5pt);
\draw[color=black] (12.35,16.55) node {$x_1$};
\fill [color=black] (8,12) circle (3.5pt);
\draw[color=black] (7.16,12.05) node {$x_3$};
\fill [color=black] (16,12) circle (3.5pt);
\draw[color=black] (16.54,12.56) node {$x_2$};
\fill [color=black] (12,8) circle (3.5pt);
\draw[color=black] (11.92,9.2) node {$x_4$};
\fill [color=black] (16,16) circle (3.5pt);
\draw[color=black] (16.84,16.35) node {$x_5$};
\fill [color=black] (16,20) circle (3.5pt);
\draw[color=black] (16.29,20.54) node {$x_6$};
\fill [color=black] (12,20) circle (3.5pt);
\draw[color=black] (12.35,20.54) node {$z_{x_6}$};
\fill [color=black] (24,16) circle (3.5pt);
\draw[color=black] (24.35,16.55) node {$x_1$};
\fill [color=black] (28,12) circle (3.5pt);
\draw[color=black] (28.51,12.56) node {$x_2$};
\fill [color=black] (20,12) circle (3.5pt);
\draw[color=black] (19.0,12.36) node {$x_3$};
\fill [color=black] (24,8) circle (3.5pt);
\draw[color=black] (24.01,9.38) node {$x_4$};
\fill [color=black] (32,12) circle (3.5pt);
\draw[color=black] (32.29,12.56) node {$x_3$};
\fill [color=black] (36,16) circle (3.5pt);
\draw[color=black] (36.4,16.55) node {$x_1$};
\fill [color=black] (40,12) circle (3.5pt);
\draw[color=black] (40.35,12.56) node {$x_2$};
\fill [color=black] (36,8) circle (3.5pt);
\draw[color=black] (35.97,9.38) node {$x_4$};
\fill [color=black] (40,16) circle (3.5pt);
\draw[color=black] (40.35,16.55) node {$x_5$};
\end{scriptsize}
\end{tikzpicture}
\end{figure}
Clearly, $\G$ is a vertex decomposable graph. Since 
$\B_{\G \setminus N_{\G}[x_6]}$ is not vertex decomposable, 
by Theorem \ref{suff-cond1},
$J(\G)^{(k)}$ is not componentwise linear for all $k \geq 2$. But 
$\B_{(G \setminus S) \setminus N_{G \setminus S}[A]}$ is a 
vertex decomposable graph for any independent set $A$ of $G \setminus S$.
\end{example}

\section{Symbolic powers of cover ideals of $W$-graphs}\label{symbolic}
In this section, 
we give a sufficient and necessary conditions so that 
$J(G)^{(k)}$ is componentwise linear for all $k \geq 2$ when 
$G$ is a $W$-graph.
The following theorem due to Woodroofe is used repeatedly throughout this paper:
\begin{theorem}\cite[Corollary 5.5]{Russ11}\label{russ-rs}
 If $G$ is a graph such that $G \setminus N_G[A]$
 has a simplicial vertex for any independent set $A$ of $G$, then $G$ is vertex decomposable.
\end{theorem}
We call a graph $G$  a \textit{$W$- graph} \footnote{The $W$-graph is named in 
 honor of \textit{Russ Woodroofe}.} if $G$ satisfies the hypothesis of 
Theorem \ref{russ-rs}.

The following observation shows that the hypothesis of Theorem \ref{russ-rs} 
includes the empty set.

\begin{obs}
%It may noted that the hypothesis of Theorem \ref{russ-rs} includes the empty set.
 %For example, 
 Let $G$ be a cycle with  vertices $x_1,\ldots,x_6$ and 
 $A$ be an independent set of $G$.
 If $A=\{x_i\}$ for some $1 \leq i \leq 6$, then $G \setminus N_G[A]$ is tree with 3 vertices. 
 Therefore $G \setminus N_G[A]$ has a  simplicial vertex.
 If $|A|>1$, then  $G \setminus N_G[A]$ is either an 
 empty graph or totally disconnected graph (with no edges).
 Therefore, $G \setminus N_G[A]$ has a  simplicial vertex
 for any non-empty independent set $A$ of $G$.
 But $G$ is not vertex decomposable \cite[Theorem 10]{Wood2}.
 Therefore, the hypothesis of Theorem \ref{russ-rs} includes the empty set.
Hence, if $G$ is a $W$-graph, then 
$G$ has a simplicial vertex. 
 \end{obs}

\begin{remark}\label{wgraphs}
 The family of $W$-graphs is rich, since it includes many interesting classes of graphs.
For example,
\begin{enumerate}
 \item if $G$ is a vertex decomposable bipartite graph, then by 
   Theorem \ref{adam-vila-rs}, 
  $G$ is a $W$-graph.
  \item In \cite{Dirac61}, Dirac proved that a graph $G$ 
   is chordal if and only if every induced subgraph
of $G$ has a simplicial vertex. Therefore, if $G$ is a chordal graph, then 
$G$ is a $W$-graph.
\item If $G$ is a simplicial graph (each of its vertices is either simplicial or is 
 adjacent to a simplicial vertex), then  $G$ is a $W$-graph.
 
 \item If $G=H \cup W(S)$, where $S \subseteq V(H)$, 
 is a graph such that $H \setminus S$ is chordal,  then 
  $G$ is a $W$- graph.
  
  \item If $G$ is a Cohen-Macaulay and very well-covered graph, 
 by \cite[Theorem 3.6]{crupi}, 
 $G$ has a  simplicial vertex and is therefore a $W$-graph.
 
 \item Let $G$ be a graph without 3-cycles and 5-cycles.
 If $G$ is a vertex decomposable graph, then by \cite[Lemma 45]{CCR16}, $G$ has a 
 simplicial vertex. Therefore, if $G$ is a vertex decomposable graph 
 without 3-cycles and 5-cycles, then  $G$ is a $W$-graph. 
\end{enumerate}
\end{remark}

We are now ready to prove  one of the  main results of this paper.

\begin{theorem}\label{main}
 Let $G$ be a $W$-graph. Then the following are
equivalent:
\begin{enumerate}
 \item $\B_{G\setminus N_G[A]}$ is a vertex decomposable graph
 for any independent set $A$ of $G$;
 \item $J(G)^{(k)}$ has linear quotients for all $k \geq 1$;
 \item $J(G)^{(k)}$ is a 
 componentwise linear ideal for all $k \geq 1$;
 \item $J(G)^{(k)}$ is a 
 componentwise linear ideal for some $k \geq 2$.
\end{enumerate}
\end{theorem}
\begin{proof}
We note that the implication $(3) \Rightarrow (4)$
is trivial and the 
implications $(2) \Rightarrow (3)$ and $(4) \Rightarrow (1)$ follow from 
Theorem \ref{lq-rs} and 
Theorem \ref{suff-cond1}.
Therefore, we only have to prove that $(1)\Rightarrow (2)$.

Suppose $\B_{G\setminus N_G[A]}$ is  vertex decomposable 
 for any independent set $A$ of $G$. We claim that $G_k$ is a $W$-graph
 for all $k \geq 1$. We prove this by induction on $k+|V(G)|$. If 
 $k=1$ and $|V(G)| \geq 2$, then $G$ is a $W$-graph. If $k \geq 1$ and 
 $|V(G)|=2$, then by \cite[Theorem 3.6]{selva1}, $G_k$ is a vertex decomposable
 bipartite graph. Therefore,
 $G_k$ is a $W$-graph for all $k \geq 1$.  
 Now assume that,
 $k \geq 2$ and $|V(G)| \geq 3$. 
 Let $B$ be an independent set of 
 $G_k$. Suppose $B=\emptyset$. Since $G$ has a simplicial vertex, by 
 \cite[Lemma 3.1]{selva1}, $G_k$ has a simplicial vertex. 
 
 Assume that 
 $B \neq \emptyset$. Let 
 $V(G)=\{x_1,\ldots,x_n\}$ and  $$B=B_1 \coprod B_2 \coprod B_3,$$
 where
 \begin{align*}
  B_1 &\subseteq \{x_{i,1} \mid 1 \leq i \leq n\},\\
  B_2 &\subseteq \{x_{i,j} \mid 1 \leq i \leq n,~ 2 \leq j \leq k-1 \}, \text{ and }\\
  B_3 &\subseteq \{x_{i,k} \mid 1 \leq i \leq n\}.
 \end{align*}
 \textsc{Case I:} Suppose $B_1 \neq \emptyset$.
 \vskip 0.1mm \noindent
 It follows from
 \cite[Lemma 3.3(3)]{selva1} that
 \[
  G_k \setminus N_{G_k}[B_1] =(G \setminus N_G[x_{i_1},\ldots,x_{i_a}])_k
  \cup \{\text{ isolated vertices }\}
 \text{ where } B_1=\{x_{i_1,1},\ldots,x_{i_a,1}\}.\]
 Since $\{x_{i_1},\ldots,x_{i_a}\}$ is an 
independent set of $G$, $H=G \setminus N_G[x_{i_1},\ldots,x_{i_a}]$ is a 
$W$-graph. 
If $A'$ is any independent set of $H$, then $A' \cup \{x_{i_1},\ldots,x_{i_a}\}$ is an 
independent set of $G$.
Since $\B_{G \setminus N_G[A' \cup \{x_{i_1},\ldots,x_{i_a}\}]}$ is a
vertex decomposable graph and $$G\setminus N_G[A' \cup \{x_{i_1},\ldots,x_{i_a}\}]
=H\setminus N_H[A'],$$ 
we have $\B_{H \setminus N_H[A']}$ is a vertex decomposable graph.
Therefore, by induction on $k+|V(G)|$, $H_k$ is a $W$-graph for all 
$k \geq 1$, i.e., $H_k \setminus N_{H_k}[C]$ has a simplicial vertex 
for any independent set $C$ of $H_k$.
Note that $$G_k \setminus N_{G_k}[B]=H_k \setminus N_{H_k}[B_2 \coprod B_3] \cup 
\{\text{ isolated vertices }\}.$$
Since $B_2 \coprod B_3$ is an independent set of 
$H_k$, we have $G_k \setminus N_{G_k}[B]$ has a simplicial vertex.
\vskip 2mm
\noindent
\textsc{Case II:} Suppose $B_1= \emptyset$. 
\vskip 0.1mm
\noindent
Since $N_{G_k}(x_{1,k},\ldots,x_{n,k})=
\{x_{1,1},\ldots,x_{n,1}\}$, 
$B'=B_2 \coprod \{x_{1,k},\ldots,x_{n,k}\}$ is an independent set 
of $G_k$. It follows from \cite[Lemma 3.3(2)]{selva1} that
\begin{align*}
 G_k \setminus N_{G_k}[B']=L \setminus N_{L}[B_2], \text{ where $L=G_k \setminus \{x_{1,1},\ldots,x_{n,1},x_{1,k},\ldots,x_{n,k}\}$.}
\end{align*}
Note that $L \simeq G_{k-2}$. 

Since $B_2$ is an independent set of 
$L$, by induction on $k+|V(G)|$, $G_k \setminus N_{G_k}[B']$ has a simplicial vertex, say $z_{p,q}$.
Set $$\mathcal{L}=G_k \setminus N_{G_k}[B'] \text{ and } \mathcal{K} =
G_k \setminus N_{G_k}[B_2 \coprod B_3].$$

Now, we have to prove that $\mathcal{K}$ has a simplicial
vertex. 
First we claim that if 
$x_{\alpha,\beta}, x_{\gamma,\delta} \in N_{\mathcal{L}}(z_{p,q})$, then
$\beta=\delta$. Assume that $\beta\neq \delta$.
Without loss of generality, let $\beta < \delta$.
Since $x_{\alpha,\beta}, x_{\gamma,\delta} \in N_{\mathcal{L}}(z_{p,q})$,
by Construction \ref{construction}, 
$\{x_{\alpha},z_p\}, \{x_{\gamma},z_p\} \in E(G)$ and 
$\beta+q \leq k+1$, $\delta+q \leq k+1$.
Then $\{x_{\alpha,\delta}, z_{p,q}\} \in E(\mathcal{L})$.
This is a contradiction to the fact that $z_{p,q}$ is a simplicial vertex
since no edges exist between $x_{\alpha,\beta}$ and $x_{\alpha,\delta}$  
and thus this neighborhood is not a clique.
%If $x_{\alpha, \beta+1},\ldots,x_{\alpha,\delta} \in B_2$, then 
%$z_{p,q} \in N_{\mathcal{L}}[B_2]$. This is a contradiction to 
%$z_{p,q}$ is a simplicial vertex of $\mathcal{L}$. Therefore, 
%$x_{\alpha, \beta+1},\ldots,x_{\alpha,\delta} \notin B_2$.
%Since $z_{p,q}$ is a simplicial vertex, we have 
%$x_{\alpha, \beta+1},\ldots,x_{\alpha,\delta} \notin N_{\mathcal{L}}(z_{p,q})$.
%Thus $x_{\alpha, \beta+1},\ldots,x_{\alpha,\delta} \in N_L[B_2]$.  Therefore
%$x_{\alpha, \beta} \in N_L[B_2]$. This is a contradiction to $x_{\alpha,\beta} \in 
%N_{\mathcal{L}}(z_{p,q})$. 
Hence the claim. 

Set 
$$N_{\mathcal{L}}(z_{p,q})=\{x_{\alpha_1,\beta},\ldots,x_{\alpha_t,\beta}\}.$$
Now, we claim that 
\begin{align}\label{main:eq}
\{x_{\alpha_1,\beta},\ldots,x_{\alpha_t,\beta}\} \subseteq 
 N_{\mathcal{K}}(z_{p,q}) \subseteq \{x_{\alpha_1,\beta},\ldots,x_{\alpha_t,\beta}, 
 x_{\alpha_1,1},\ldots,x_{\alpha_t,1}\}. 
\end{align}
Since $\mathcal{L}$ is a subgraph of $\mathcal{K}$, we have
 $\{x_{\alpha_1,\beta},\ldots,x_{\alpha_t,\beta}\}\subseteq N_{\mathcal{K}}(z_{p,q})$.
 Suppose $x_{\lambda,1} \in N_{\mathcal{K}}(z_{p,q})$ for some 
$\lambda \notin \{\alpha_1,\ldots,\alpha_t\}$.
Since $i+q \leq k+1$ for all $1 \leq i \leq \lfloor \frac{k+1}{q} \rfloor$,
by Construction \ref{construction},
$\{x_{\lambda,i}, z_{p,q}\} \in E(G_k)$ for all 
$1 \leq i \leq \lfloor \frac{k+1}{q} \rfloor$.
Since $z_{p,q}$ is a simplicial vertex of 
$\mathcal{L}$, we have
\begin{align*}
 x_{\lambda,i} \notin N_{\mathcal{L}}(z_{p,q})
 \text{ for all } 
 2 \leq i \leq \lfloor \frac{k+1}{q} \rfloor &\Rightarrow
 x_{\lambda,i} \in N_{G_k}[B'] \text{ for all }  
 2 \leq i \leq \lfloor \frac{k+1}{q} \rfloor,\\
 &\Rightarrow x_{\lambda,i} \in N_{L}[B_2] \text{ for all }  
 2 \leq i \leq \lfloor \frac{k+1}{q} \rfloor,\\
 &\Rightarrow x_{\lambda,i} \in N_{G_k}[B_2] \text{ for all }  
 1 \leq i \leq \lfloor \frac{k+1}{q} \rfloor.
\end{align*}
 Therefore
 $x_{\lambda,1} \notin V(\mathcal{K})$.
 Hence $N_{\mathcal{K}}(z_{p,q}) \subseteq \{x_{\alpha_1,\beta},\ldots,x_{\alpha_t,\beta}, 
 x_{\alpha_1,1},\ldots,x_{\alpha_t,1}\}$.

 If $x_{\alpha_j,1} \notin N_{\mathcal{K}}(z_{p,q})$ for all $1 \leq j \leq t$, 
 then by \eqref{main:eq}, $N_{\mathcal{K}}(z_{p,q})=
 \{x_{\alpha_1,\beta},\ldots,x_{\alpha_t,\beta}\}$. 
  Note that $\{x_{\alpha_1,\beta},\ldots,x_{\alpha_t,\beta}\}$
 is a clique in $\mathcal{K}$, i.e., $\{x_{\alpha_1},\ldots,x_{\alpha_t}\}$
is a clique in $G$. 
 Therefore, $z_{p,q}$ is a 
 simplicial vertex of $\mathcal{K}$.  
 Suppose 
 $x_{\alpha_j,1} \in N_{\mathcal{K}}(z_{p,q})$ for some $1 \leq j \leq t$.
Since $z_{p,k}$ is only connected to vertices of the form 
$x_{\nu,1}$ for some $x_{\nu} \in V(G)$, we have by \eqref{main:eq}, $N_{\mathcal{K}}(z_{p,k}) 
 \subseteq \{x_{\alpha_1,1},\ldots,x_{\alpha_t,1}\} $. Then
 $z_{p,k}$
 is a simplicial vertex of $\mathcal{K}$. 
 Therefore $\mathcal{K}$ has a simplicial vertex.
 Hence, by Theorem \ref{russ-rs}, $G_k$ is a vertex decomposable graph for all 
 $k \geq 2$.
 
 By \cite[Proposition 8.2.5]{Herzog'sBook}, $J(G_k)$ has linear quotients for all 
$k \geq 1$.
Therefore, by  Lemma \ref{fak-result}, Lemma \ref{pol_reg}, 
$ J(G)^{(k)}$ has linear quotients.
%Hence $J(G)^{(k)}$ is componentwise linear for all $k \geq 1$.
%
%\vskip 1mm
%\noindent
%$(2)\Rightarrow (3)$ This implication is trivial.
%$(3)\Rightarrow (1)$ This implication follows from Theorem \ref{suff-cond1}.
 \end{proof}
\begin{remark}
It is interesting to note that 
if $\B_{G \setminus N_G[A]}$ is a vertex decomposable graph for any
non-empty independent set $A$ of $G$,
then $\B_G$ is not necessarily vertex decomposable.
For example, let $G$ 
be a graph as shown in 
Example \ref{ex}. It is not hard to verify that $\B_{G \setminus N_G[A]}$ 
is vertex decomposable for all non-empty independent sets $A$ of $G$ but 
$\B_G$ is not vertex decomposable.
If $\B_G$ is a vertex decomposable graph, then 
$\B_{G \setminus N_G[A]}$ is not necessarily vertex decomposable for any non-empty 
independent set $A$
of $G$.

 \begin{figure}[H]

\begin{tikzpicture}[scale=0.33]
%\clip(0.73,-1.5) rectangle (50.08,25.17);
\draw (6,19.97)-- (6,18);
\draw (6,18)-- (4,16);
\draw (8,16)-- (6,18);
\draw (4,12)-- (8,12);
\draw (4,16)-- (4,12);
\draw (8,16)-- (8,12);
\draw (4,16)-- (8,16);
\draw (8,8)-- (8,12);
\draw (12,12)-- (8,8);
\draw (4,12)-- (8,8);
\draw (8,12)-- (12,12);
\draw (4,16)-- (8,12);
\draw (8,16)-- (4,12);
\draw (16,12)-- (16,16);
\draw (16,16)-- (20,16);
\draw (16,16)-- (20,12);
\draw (16,16)-- (18,18);
\draw (18,20)-- (18,18);
\draw (20,8)-- (20,12);
\draw (16,12)-- (20,8);
\draw (24,12)-- (20,8);
\draw (28,12)-- (32,12);
\draw (32,8)-- (28,12);
\draw (32,12)-- (32,8);
\draw (36,12)-- (32,8);
\draw (36,12)-- (32,12);
\draw (40,12)-- (44,12);
\draw (44,8)-- (40,12);
\draw (48,12)-- (44,8);
\draw (48,12)-- (44,12);
\draw (4.08,7.5) node[anchor=north west] { Fig. 1. $G$};
\draw (16.04,7.55) node[anchor=north west] { Fig. 2. $\mathcal{B}_G$};
\draw (26.18,7.55) node[anchor=north west] { Fig. 3. $G\setminus N_G[x_2] $};
\draw (39.69,7.63) node[anchor=north west] {Fig. 4. $\mathcal{B}_{G\setminus N_G[x_2]} $};
\begin{scriptsize}
\fill [color=black] (6,19.97) circle (3.5pt);
\draw[color=black] (6.6,20.54) node {$x_1$};
\fill [color=black] (6,18) circle (3.5pt);
\draw[color=black] (6.7,18.56) node {$x_2$};
\fill [color=black] (4,16) circle (3.5pt);
\draw[color=black] (3.19,16.21) node {$x_3$};
\fill [color=black] (8,16) circle (3.5pt);
\draw[color=black] (8.62,16.55) node {$x_4$};
\fill [color=black] (4,12) circle (3.5pt);
\draw[color=black] (3.13,12.13) node {$x_5$};
\fill [color=black] (8,12) circle (3.5pt);
\draw[color=black] (8.62,12.56) node {$x_6$};
\fill [color=black] (8,8) circle (3.5pt);
\draw[color=black] (7.15,8.06) node {$x_7$};
\fill [color=black] (12,12) circle (3.5pt);
\draw[color=black] (12.61,12.56) node {$x_8$};
\fill [color=black] (16,12) circle (3.5pt);
\draw[color=black] (15.14,12.05) node {$x_5$};
\fill [color=black] (16,16) circle (3.5pt);
\draw[color=black] (15.14,16.29) node {$x_3$};
\fill [color=black] (20,16) circle (3.5pt);
\draw[color=black] (20.28,16.55) node {$x_4$};
\fill [color=black] (20,12) circle (3.5pt);
\draw[color=black] (20.28,12.56) node {$x_6$};
\fill [color=black] (18,18) circle (3.5pt);
\draw[color=black] (18.79,18.56) node {$x_2$};
\fill [color=black] (18,20) circle (3.5pt);
\draw[color=black] (18.35,20.54) node {$x_1$};
\fill [color=black] (20,8) circle (3.5pt);
\draw[color=black] (19.12,8.06) node {$x_7$};
\fill [color=black] (24,12) circle (3.5pt);
\draw[color=black] (24.35,12.56) node {$x_8$};
\fill [color=black] (28,12) circle (3.5pt);
\draw[color=black] (28.34,12.56) node {$x_5$};
\fill [color=black] (32,12) circle (3.5pt);
\draw[color=black] (32.33,12.56) node {$x_6$};
\fill [color=black] (32,8) circle (3.5pt);
\draw[color=black] (31.0,8.15) node {$x_7$};
\fill [color=black] (36,12) circle (3.5pt);
\draw[color=black] (36.32,12.56) node {$x_8$};
\fill [color=black] (40,12) circle (3.5pt);
\draw[color=black] (40.35,12.56) node {$x_5$};
\fill [color=black] (44,12) circle (3.5pt);
\draw[color=black] (44.33,12.56) node {$x_6$};
\fill [color=black] (44,8) circle (3.5pt);
\draw[color=black] (43.05,8.1) node {$x_7$};
\fill [color=black] (48,12) circle (3.5pt);
\draw[color=black] (48.36,12.56) node {$x_8$};
\end{scriptsize}
\end{tikzpicture}
\end{figure}

 For example, let $G$ be a graph as shown in Fig. 1.
Since $G$ is a chordal graph, $G$ is a $W$-graph. 
One can see that 
$\mathcal{B}_G$ is a vertex decomposable bipartite graph but $\B_{G \setminus N_G[x_2]}$ 
is not vertex decomposable.  
\end{remark}

The following result partially answers a question asked in \cite[Question 4.9]{GTS20}. 
\begin{corollary}\label{main-cor}
 Let $\G=G \cup W(S)$ be a $W$-graph and $S \subseteq V(G)$.
  If $\B_{\G \setminus N_{\G}[A]} \setminus S'$ is a forest for 
 any independent set $A$ of $\G$, where $S' = S \setminus N_{\G}[A]$, then 
 $J(\G)^{(k)}$ is a componentwise linear ideal for all $k \geq 2$.
\end{corollary}
\begin{proof}
Let $\G$ be a graph with vertex set $V(G) \cup \{z_x \mid x \in S\}$ and 
edge set $$E(G) \cup \{ \{x,z_x\} \mid x \in S\}.$$
Observe that $\deg_{\B_\G \setminus N_{\G}[A]}(z_{x})=1$ for all $x \in S'$.
 Since $\B_{\G \setminus N_{\G}[A]} \setminus S'$ is a forest,  
 by \cite[Corollary 4.8]{BFH15},
 $\B_{\G \setminus N_{\G}[A]}$ is a vertex decomposable graph. Therefore,
 by Theorem \ref{main}, $J(\G)^{(k)}$ is componentwise linear for all $k \geq 2$.
\end{proof}
\begin{remark}
Note that if $G \setminus S$ is a forest, then $\B_{\G \setminus N_{\G}[A]} \setminus S'$ is a forest for 
 any independent set $A$ of $\G$. Therefore, 
Corollary \ref{main-cor} allows us to quickly recover and extend for instance
\cite[Theorem 5.7]{DHNT20}, \cite[Corollary 4.5]{selva1}, \cite[Corollary 4.4]{KK20},
\cite[Theorem 3.11]{GTS20}.
\end{remark}
\iffalse
Mohammadi, Kiani and Yassemi proved that if $G$ is a unicyclic graph with cycle
$C_n$, $n \neq 3, 5$, then $G$ is vertex decomposable if and only 
if at least one whisker is attached to $C_n$ (\cite{FDY10}). 
%Therefore, $G$ is a $W$-graph. 
Note that 
if $G$ is a unicyclic graph with cycle
$C_n$, $n \neq 3, 5$ and vertex decomposable
graph, then $G=H \cup W(z)$, where $z$ is the vertex 
on cycle, is a graph such that $H \setminus z$ is a forest.
We now derive the main results of Kumar and Kumar (\cite[Corollary 4.4]{KK20}) and 
Gu, H\`a and Skelton (\cite[Theorem 3.11]{GTS20}).

\begin{corollary}\label{recover-rs} 
If $G\cup W(S)$, $S \subseteq V(G)$ is a graph such that 
$G \setminus S$ forest, then $J(G \cup W(S))^{(k)}$ is a 
componentwise linear ideal for all $k \geq 2$.

\end{corollary}
\begin{proof}
Let $\G=G \cup W(S)$ and $V(\G)=V(G) \cup \{z_x \mid x \in S\}$.
Note that 
 $\G \setminus N_{\G}[A]=H \cup W(S')$, where
 $H$ is an induced subgraph of $G$ and $S' \subseteq S$,
 and $H \setminus S'$ is a forest for any independent set 
 $A$ of $\G$. Therefore, by Theorem \ref{main}, it is enough to prove that
 $\B_{\G}$ is vertex decomposable. 
 Observe that $\deg_{\B_\G}(z_{x})=1$ for all $x \in S$.
 Since $G \setminus S$ is a forest,  
 $\B_\G \setminus S$ is a forest. 
 Therefore, \cite[Corollary 4.6]{BFH15},
 $\B_\G$ is a vertex decomposable graph.
 \end{proof}

\fi
In \cite{FH}, Francisco and H\`a proved that if 
   $|S| \geq |V(G)|-3$, where $S \subseteq V(G)$, 
   then $J(G \cup W(S))$ is componentwise linear.
   We extend this result to include all symbolic powers of $J(G \cup W(S))$.
 \begin{corollary}\label{3-ver}
   Let $G$ be a graph and $S \subseteq V(G)$. If 
   $|S| \geq |V(G)|-3$, then $J(G \cup W(S))^{(k)}$ is a componentwise linear ideal
   for all $k \geq 1$.
  \end{corollary}
  \begin{proof}
  Set 
   $\G=G \cup W(S)$.
   Since $|S| \geq |V(G)| - 3$, we have 
   $G \setminus S$ is a graph on at most 3 vertices. Thus, 
   $G \setminus S$ is either a three-cycle, a tree, or set of  isolated vertices. 
   Therefore,  $\B_{\G \setminus N_{\G}[A]} \setminus S'$ is a forest for 
 any independent set $A$ of $\G$, where $S'=S \setminus N_G[A]$. By Corollary \ref{main-cor},
 $J(\G)^{(k)}$ is a componentwise linear ideal for all $k \geq 2$.
  \end{proof}
  %\begin{remark}
 %We would like to note here that, Corollary \ref{3-ver} is a generalization of 
 %\cite[Theorem 3.11]{GTS20} and \cite[Corollary 4.5]{selva1} when 
 %$G \setminus S$ is a three-cycle.
 % \end{remark}

  We recall the definition of star graph based on a complete graph from \cite[page 7]{HerHibiOhsugi}.
\begin{definition}\label{star-graph}
 We say that $G$ is a star graph based on a complete graph $K_n$ if $G$ is connected and 
 $V(G)=\{x_1,\ldots,x_n,y_1,\ldots,y_m\}$ such that:
 \begin{enumerate}
  \item the complete graph on $\{x_1,\ldots,x_n\}$ is a subgraph of $G$ and
  \item there is no edge in $G$ connecting $y_i$ and $y_j$ for all $1 \leq i< j \leq m$.
 \end{enumerate}
\end{definition}

We make some remarks which follow directly from 
Definition \ref{star-graph} and Construction \ref{intr-cons}.
\begin{remark}\label{star-obs}
Let $G$ be a star graph based on a complete graph $K_n$ with 
$V(G)=$ $\{x_1,\ldots,x_n$, $y_1,\ldots,y_m\}$. 
\begin{enumerate}
 \item Since $G$ is a chordal graph, $G$ is a 
$W$-graph.
If $|N_G(y_1,\ldots,y_m)| \leq n-1$, say $x_1 \notin N_G(y_1,\ldots,y_m)$, then 
$(x_2,\ldots,x_n)$ is a shedding order of $G$ and 
$\{x_1,y_1,\ldots,y_m\}$ is an $i$-order of 
$G$. Suppose $|N_G(y_1,\ldots,y_m)| =n$. Let $N_{G \setminus N_G(y_1,\ldots,y_{i-1})}(y_i)=
\{x_{i1},\ldots,x_{it_i}\}$ for all $1 \leq i \leq m$.
Then $(x_{11},\ldots,x_{1t_1},\ldots,x_{m1},\ldots,x_{mt_m})$ is a shedding order of $G$
and $\{y_1,\ldots,y_m\}$ is an $i$-order of $G$.
Note that $\{x_1,\ldots,x_n\}=\{x_{11},\ldots,x_{1t_1},\ldots,x_{m1},\ldots,x_{mt_m}\}$ and 
$n=t_1+\ldots+t_m$.
\item Let $A$ be an independent set of $G$. If $A \subseteq \{y_1,\ldots,y_m\}$, then 
we have $\B_{G \setminus N_G[A]}=\B_{G} \setminus N_{\B_G}[A].$
If $x_i \in A$ for some $1 \leq i \leq n$, then $G \setminus N_G[A]$ is 
totally disconnected. Therefore, if 
$\B_G$ is a vertex decomposable graph, then 
 $\B_{G \setminus N_{G}[A]}$ is vertex decomposable for any non-empty independent set 
 $A$ of $G$.
\end{enumerate}
\end{remark}

As an immediate consequence of Theorem \ref{main} and Remark \ref{star-obs},
we give an answer to \cite[Question 5.13]{DHNT20}.

\begin{corollary}\label{star-main}
Let $G$ be a star graph based on a complete graph $K_n$.
 Then the following are equivalent:
 \begin{enumerate}
  \item $\B_G$ is a vertex decomposable graph;
   \item  $J(G)^{(k)}$ has linear quotients for all $k \geq 1$;
  \item $J(G)^{(k)}$ is a componentwise linear ideal for all $k \geq 2$;
  \item $J(G)^{(k)}$ is a componentwise linear ideal for some $k \geq 2$.
 \end{enumerate}
 \end{corollary}

 The following remark shows that there is a one-to-one correspondence between
 the family of  star graphs based on a complete graph $K_n$ 
 and the family of bipartite graphs.
 
\begin{remark}
If $G$ is a star graph based on a complete graph $K_n$ with $V(G)$ $=\{x_1,\ldots,x_n$,
 $y_1,\ldots,y_m\}$ and $|N_G(y_1,\ldots,y_m)|=n$, then we can associate the 
 bipartite graph $B_G$ with partitions 
 $\{x_1,\ldots,x_n\} \sqcup \{y_1,\ldots,y_m\}$ and 
 $$E(B_G)=\{\{x_i,y_j\} \in E(G) \mid 1 \leq i \leq n,~1\leq j \leq m\}.$$
 Note that $(x_1,\ldots,x_n)$ is a shedding order of $G$ and 
 $\{y_1,\ldots,y_m\}$ is an $i$-order of $G$. Therefore, by Construction \ref{intr-cons},
 $B_G=\B_G$.  
 
 If $B$ is a bipartite graph with partitions $\{x_1,\ldots,x_n\} \sqcup 
 \{y_1,\ldots,y_m\}$ and no isolated vertices, then we can associate 
 the star graph $H$ based on a complete graph $K_n$ with $V(H)$ $=\{x_1,\ldots,x_n$,
 $y_1,\ldots,y_m\}$ and $|N_H(y_1,\ldots,y_m)|=n$. 
 It follows from Construction \ref{intr-cons} that 
 $B=\B_H$. 
 
 Therefore, the assignment $G \longrightarrow \B_G$ establishes a natural 
 one-to-one correspondence  between the family of  
 star graphs based on a complete graph $K_n$ with vertices 
 $\{x_1,\ldots,x_n$, $y_1,\ldots,y_m\}$ and $|N_G(y_1,\ldots,y_m)|=n$ and the 
 family of bipartite graphs with partitions
 $\{x_1,\ldots,x_n\}$ $\sqcup$ $\{y_1,\ldots,y_m\}$ and no isolated vertices
 up to isomorphism.

\end{remark}

We recall the definition from \cite{HHM20}.

\begin{definition} \label{n-clique-def}
 Let $V=\{x_1,\ldots,x_p\} \cup \{y_{ij} \mid 1 \leq i \leq n,~1 \leq j \leq m_i\}$ be a
 finite set. We write $G_{p,m_i}$ for the complete graph  on
 $\{x_1,\ldots,x_p,y_{i1},\ldots,y_{im_i}\}$. Let 
 $\Gamma_{p,m_1,\ldots,m_n}$ be a graph with $V(\Gamma_{p,m_1,\ldots,m_n})=V$ and 
 $E(\Gamma_{p,m_1,\ldots,m_n})=\bigcup \limits_{1 \leq i \leq n} E(G_{p,m_i})$.
 We call $\Gamma_{p,m_1,\ldots,m_n}$ is an $n$-clique graph.
 \end{definition}

Below are 2 examples of $n$-clique graphs.

\begin{minipage}{\linewidth}
  %\centering
\begin{minipage}{0.36\linewidth}

\begin{figure}[H]
\begin{tikzpicture}[scale=0.6]
\draw (3,3)-- (2,2);
\draw (3,3)-- (4,2);
\draw (2,2)-- (4,2);
\draw (2,4)-- (3,3);
\draw (4,4)-- (3,3);
\draw (3,5)-- (2,4);
\draw (3,5)-- (4,4);
\draw (3,5)-- (3,3);
\draw (2,4)-- (4,4);
\draw (3,3)-- (4,3);
\draw (8,2)-- (7,3);
\draw (9,3)-- (8,2);
\draw (7,3)-- (9,3);
\draw (7.02,4.26)-- (7,3);
\draw (8.96,4.24)-- (9,3);	
\draw (7.02,4.26)-- (8.96,4.24);
\draw (8.96,4.24)-- (7,3);
\draw (7.02,4.26)-- (9,3);
\draw (8.96,4.24)-- (8,2);
\draw (7.02,4.26)-- (8,2);
\draw (7.04,5.44)-- (9,5.4);
\draw (7.04,5.44)-- (7.02,4.26);
\draw (9,5.4)-- (8.96,4.24);
\draw (9,5.4)-- (7.02,4.26);
\draw (8.96,4.24)-- (7.04,5.44);
\draw (2.66,1.66) node[anchor=north west] {$ \Gamma_{1,3,2,1} $};
\draw (7.64,1.62) node[anchor=north west] {$\Gamma_{2,2,3}$};
\begin{scriptsize}
\fill [color=black] (3,3) circle (1.5pt);
\draw[color=black] (2.4,3.02) node {$x_1$};
\fill [color=black] (2,2) circle (1.5pt);
\draw[color=black] (1.5,2.06) node {$y_{21}$};
\fill [color=black] (4,2) circle (1.5pt);
\draw[color=black] (4.56,2.26) node {$y_{22}$};
\fill [color=black] (2,4) circle (1.5pt);
\draw[color=black] (1.5,4.06) node {$y_{11}$};
\fill [color=black] (4,4) circle (1.5pt);
\draw[color=black] (4.36,4.26) node {$y_{12}$};
\fill [color=black] (3,5) circle (1.5pt);
\draw[color=black] (3.14,5.26) node {$y_{13}$};
\fill [color=black] (4,3) circle (1.5pt);
\draw[color=black] (4.16,3.26) node {$y_{31}$};
\fill [color=black] (8,2) circle (1.5pt);
\draw[color=black] (8.54,1.84) node {$y_{22}$};
\fill [color=black] (7,3) circle (1.5pt);
\draw[color=black] (6.5,3.04) node {$y_{23}$};
\fill [color=black] (9,3) circle (1.5pt);
\draw[color=black] (9.46,3.14) node {$y_{21}$};
\fill [color=black] (7.02,4.26) circle (1.5pt);
\draw[color=black] (6.68,4.34) node {$x_1$};
\fill [color=black] (8.96,4.24) circle (1.5pt);
\draw[color=black] (9.3,4.5) node {$x_2$};
\fill [color=black] (7.04,5.44) circle (1.5pt);
\draw[color=black] (7.22,5.7) node {$y_{11}$};
\fill [color=black] (9,5.4) circle (1.5pt);
\draw[color=black] (9.16,5.66) node {$y_{12}$};
\end{scriptsize}
\end{tikzpicture}
	\end{figure}
\end{minipage}
\begin{minipage}{0.59\linewidth}

 In \cite[Corollary 4.7]{HHM20}, Herzog, Hibi, and Moradi proved that 
 if $G=\Gamma_{p,m_1,m_2}$, then $J(G)^k$ is a componentwise linear ideal for all
 $k \geq 1$. Hence it is natural to ask: If $G=\Gamma_{p,m_1,m_2}$, is it true that
 $J(G)^{(k)}$ is componentwise linear for all $k \geq 2$? The answer is "No". For example
 if $G=\Gamma_{2,1,1}$, then 
 $\B_G$ is a cycle with length 4. Therefore, by Theorem \ref{suff-cond1}, 
 $J(G)^{(k)}$ is not componentwise linear for all 
 $k \geq 2$. 
\end{minipage}
\end{minipage}
Now, we classify the graph $G=\Gamma_{p,m_1,\ldots,m_n}$
 such that $J(G)^{(k)}$ is a componentwise linear ideal for all $k \geq 2$. 
 Let $G=\Gamma_{p,m_1,\ldots,m_n}$ be a graph. If $p=1$, then 
$G$ is a star complete graph (see the definition in \cite{selva1}). 
Therefore, by \cite[Theorem 4.2]{selva1},
$J(G)^{(k)}$ is a componentwise linear ideal for all $k \geq 2$.

\begin{corollary}\label{n-clique}
 Let $G=\Gamma_{p,m_1,\ldots,m_n}$ be a graph with $p>1$. Then the following are 
 equivalent:
 \begin{enumerate}
  %\item %at most one $m_i$ is equal to 1 for some $1 \leq i \leq n$;
  %\begin{enumerate}
         \item $m_i>1$ for all $1 \leq i \leq n$, or
         exactly one $m_i=1$ for some $1 \leq i \leq n$ 
         and $m_j>1$ for all $1 \leq j \neq i \leq n$;
   %\end{enumerate}
   \item $J(G)^{(k)}$ has linear quotients for all $k \geq 1$;
\item $J(G)^{(k)}$ is a componentwise linear ideal for all $k \geq 1$;
\item $J(G)^{(k)}$ is a componentwise linear ideal for some $k \geq 2$.
\end{enumerate}
\end{corollary}
\begin{proof}
Let $G$ be a graph with 
$V(G)=\{x_1,\ldots,x_p\} \cup \{y_{ij} \mid 1 \leq i \leq n,~1 \leq j \leq m_i\}$.
Since $G$ is a chordal graph, $G$ is a $W$-graph.
 By Theorem \ref{main}, it is enough to prove the following are equivalent:
 \begin{enumerate}
  
\item[(i)] $m_i>1$ for all $1 \leq i \leq n$, or
         exactly one $m_i=1$ for some $1 \leq i \leq n$ 
         and $m_j>1$ for all $1 \leq j \neq i \leq n$;
 \item[(ii)] $\B_{G\setminus N_G[A]}$ is a
vertex decomposable graph for any independent set $A$ of $G$.
 \end{enumerate}
(i)$\Rightarrow$ (ii)
  Let $A$ be an independent set of $G$.
  \vskip 0.1mm \noindent
  \textsc{Case I:} Suppose $A \neq \emptyset$.
  
 If $x_i \in A$ for some $1 \leq i \leq p$, then 
 $G \setminus N_{G}[A]$ is totally disconnected. 
 If $x_{i} \notin A$ for some $1 \leq i \leq p$, then 
 $G \setminus N_{G}[A]$ is a disjoint union of complete graphs.
 It follows from Construction \ref{intr-cons} that 
 $\B_{G \setminus N_G[A]}$ is a forest.  Therefore, 
 $\B_{G \setminus N_G[A]}$ is a vertex decomposable for any 
 non-empty independent set $A$ of $G$.
 \vskip 0.1mm \noindent
  \textsc{Case II:} Suppose $A = \emptyset$.
 
 Assume that  $m_i>1$ for all $1 \leq i \leq n$.
 Note that  
 $$(x_1,\ldots,x_p, y_{12},\ldots,y_{1m_1}, y_{22},\ldots,y_{2m_2}, \ldots,
 y_{n2},\ldots,y_{nm_n})$$ is a shedding order of $G$
 and $\{y_{11},y_{21},\ldots,y_{n1}\}$ is an 
 $i$-order of $G$. By Construction \ref{intr-cons},
 \begin{align*}
  E(\B_G)&=\{ \{y_{k1},x_i\}, \{y_{k1}, y_{kj}\} \mid 1 \leq i \leq p,~1 \leq k \leq n,
  ~2 \leq j \leq m_k\}.
 \end{align*}
Since $\B_G \setminus \{y_{11},\ldots,y_{n1}\}$ is totally
disconnected and $\deg_{\B_G}(y_{j2})=1$ for all $1 \leq j \leq n$,
 we have $\B_G$ is  vertex decomposable by \cite[Corollary 4.8]{BFH15}.
  
  %Therefore $\B_G \setminus N_{\B_G}[y_{12},\ldots,y_{n2}]$ is totally disconnected.
  %Hence, by \cite[Corollary 4.8]{BFH15}, $\B_G$ is  vertex decomposable.

 Suppose $m_i=1$ for some $1 \leq i \leq n$. 
 Note that  
 $$(x_1,\ldots,x_p, y_{12},\ldots,y_{1m_1}, \ldots, 
 y_{i-12},\ldots,y_{i-1m_{i-1}}, y_{i+12}, \ldots, y_{i+1m_{i+1}}, \ldots,
 y_{n2},\ldots,y_{nm_n})$$ is a shedding order of $G$
 and $\{y_{11},y_{21},\ldots,y_{n1}\}$ is an 
 $i$-order of $G$. By Construction \ref{intr-cons},
 \begin{align*}
  E(\B_G)&= \{ \{y_{k1},x_l\}, \{y_{k1}, y_{kj}\} \mid 1 \leq l \leq p,~1 \leq k \neq i 
  \leq n,
  ~2 \leq j \leq m_k\}\\
  & \bigcup \{ \{y_{i1},x_l\} \mid 1 \leq l \leq p\}.
 \end{align*} 
 Since 
 $\deg_{\B_G}(y_{j2})=1$ for all $1 \leq j \neq i \leq n$, 
 $\B_G \setminus \{y_{11},\ldots,y_{(i-1)1},y_{(i+1)1},\ldots,y_{n1}\}$ is a 
 tree. Again, by \cite[Corollary 4.8]{BFH15}, $\B_G$ is  vertex decomposable.
\vskip 1mm
\noindent
 (ii)$\Rightarrow$ (i)
 %Let $\B_{G\setminus N_G[A]}$ be a
%vertex decomposable graph for any independent set $A$ of $G$.
Suppose 
$m_l>1$ for all $l \neq i,j$ and $m_i=1=m_j$.
Then 
$$(x_1,\ldots,x_p, y_{12},\ldots,y_{1m_1}, \ldots, 
 y_{i-12},\ldots,y_{i-1m_{i-1}}, \ldots, y_{j-12}, \ldots, y_{j-1m_{j-1}}, \ldots,
 y_{n2},\ldots,y_{nm_n})$$ is a shedding order of $G$
 and $\{y_{11},y_{21},\ldots,y_{n1}\}$ is an 
 $i$-order of $G$. By Construction \ref{intr-cons},
 \begin{align*}
  E(\B_G)&= \{ \{y_{k1},x_l\}, \{y_{k1}, y_{kj}\} \mid 1 \leq l \leq p,~1 \leq k \neq i,j 
  \leq n,
  ~2 \leq j \leq m_k\}\\
  & \bigcup \{ \{y_{i1},x_l\} \mid 1 \leq l \leq p\}\bigcup 
  \{ \{y_{j1},x_l\} \mid 1 \leq l \leq p\}.
 \end{align*} 
%By Theorem \ref{ind-sub}, $\B_G \setminus N_{\B_G}[B]$ is 
%vertex decomposable for any independent set $B$.
Set $$B=\{y_{l2} \mid 1 \leq l \leq n, ~l \notin \{i,j\}\}.$$
%Since $\B_G$ is  vertex decomposable and  
%$B$ is an independent set of 
%$\B_G$, by Theorem \ref{ind-sub}, 
%$\B_G\setminus N_{\B_G}[B]$ is a vertex decomposable graph.
One can see that $\deg_{\B_G \setminus N_{\B_G}[B]}(z)>1$ for all $z \in 
V(\B_G \setminus 
N_{\B_G}[B]).$ 
Therefore, by Theorem \ref{adam-rs}, Theorem \ref{adam-vila-rs},
$\B_G \setminus N_{\B_G}[B]$ is not a vertex decomposable graph.
This is a contradiction to $\B_G$ is 
vertex decomposable bipartite. Therefore,
$m_i>1$ for all $1 \leq i \leq n$ or
         exactly one $m_i=1$ for some $1 \leq i \leq n$ 
         and $m_j>1$ for all $1 \leq j \neq i \leq n$.
 \end{proof}

 \section{Powers of cover ideals of bipartite graphs}\label{bipartite}

In this section, we prove that if $G$ is a bipartite graph, then 
$J(G)$ is a componentwise linear ideal if and only if $J(G)^k$ is a componentwise linear
ideal for some (equivalently, for all) $k \geq 2$.

The following result is a generalization of Theorem \ref{not-cl} when $G$ is a 
bipartite graph.

 \begin{theorem}\label{bi-not-cl}
 Let $G$ be a bipartite graph. If $J(G)$ is not componentwise linear ideal, then 
 $J(G)^k$ is not a componentwise linear ideal for all $k\geq 2$.
\end{theorem}
\begin{proof}
 Let $G$ be a bipartite graph with partitions $\{x_1,\ldots,x_n\}$ $\sqcup$
 $\{y_1,\ldots,y_m\}$. 
 Let $G_k$ be a graph as in Construction \ref{construction} for all $k \geq 1$.
 It follows from \cite[Corollary 2.6]{GRV05} and Lemma \ref{fak-result}
 that $\widetilde{J(G)^{(k)}}=\widetilde{J(G)^{k}}=J(G_k)$ for all $k \geq 1$.
 Note that  
 $G_i \setminus N_{G_i}[x_{1,i},\ldots,x_{n,i}] \simeq G_{i-1}$ for all
 $i \geq 1$. Now proceeding as in Theorem \ref{not-cl}, one can show that
 $G_k$ is not vertex decomposable for all $k \geq 2$. Therefore,
 $J(G)^{k}$ is not a componentwise linear  ideal
 for all $k \geq 2$.
\end{proof}

It may be noted that if $G$ is a bipartite graph, then $G$ may not be equal to 
$\B_G$. 

\begin{minipage}{\linewidth}
  %\centering
\begin{minipage}{0.45\linewidth}

\begin{figure}[H]

\begin{tikzpicture}[scale=0.3]
%\clip(-2.45,-2.04) rectangle (46.98,24.63);
\draw (6,18)-- (10,18);
\draw (6,14)-- (10,14);
\draw (6,10)-- (10,10);
\draw (10,6)-- (6,10);
\draw (10,14)-- (6,10);
\draw (10,10)-- (6,14);
\draw (10,18)-- (6,14);
\draw (14,18)-- (18,18);
\draw (14,14)-- (18,14);
\draw (14,10)-- (18,10);
\draw (18,6)-- (14,10);
\draw (18,14)-- (14,10);
\draw (14,14)-- (18,10);
\draw (4.84,4.69) node[anchor=north west] {Fig. 1. $G$};
\draw (13.6,4.78) node[anchor=north west] {Fig. 2. $ \mathcal{B}_G $};
\begin{scriptsize}
\fill [color=black] (6,18) circle (3.5pt);
\draw[color=black] (5.07,18.2) node {$x_1$};
\fill [color=black] (10,18) circle (3.5pt);
\draw[color=black] (11.01,18.2) node {$y_1$};
\fill [color=black] (6,14) circle (3.5pt);
\draw[color=black] (4.87,14.17) node {$x_2$};
\fill [color=black] (10,14) circle (3.5pt);
\draw[color=black] (10.59,14.85) node {$y_2$};
\fill [color=black] (6,10) circle (3.5pt);
\draw[color=black] (4.81,10.05) node {$x_3$};
\fill [color=black] (10,10) circle (3.5pt);
\draw[color=black] (10.59,10.86) node {$y_3$};
\fill [color=black] (10,6) circle (3.5pt);
\draw[color=black] (10.59,6.88) node {$y_4$};
\fill [color=black] (14,18) circle (3.5pt);
\draw[color=black] (14.36,18.74) node {$x_1$};
\fill [color=black] (18,18) circle (3.5pt);
\draw[color=black] (18.22,18.74) node {$y_1$};
\fill [color=black] (14,14) circle (3.5pt);
\draw[color=black] (13.02,14.17) node {$x_2$};
\fill [color=black] (18,14) circle (3.5pt);
\draw[color=black] (18.3,14.7) node {$y_2$};
\fill [color=black] (14,10) circle (3.5pt);
\draw[color=black] (13.07,10.14) node {$x_3$};
\fill [color=black] (18,10) circle (3.5pt);
\draw[color=black] (18.59,10.86) node {$y_3$};
\fill [color=black] (18,6) circle (3.5pt);
\draw[color=black] (18.55,6.98) node {$y_4$};
\end{scriptsize}
\end{tikzpicture}
	\end{figure}
\end{minipage}
\begin{minipage}{0.5\linewidth}

For example, let $G$ be a graph as shown in Fig. 1. It is not hard to see that 
$(y_1, x_3, x_2)$ is a shedding order of $G$ and $\{x_1,y_2,y_3,y_4\}$ is an $i$-order. Note that
$G$ is not equal to $\B_G$.
\end{minipage}
\end{minipage}

In the following lemma, we prove that 
if $G$ is a vertex decomposable bipartite graph, then so is $\B_G$.
This proof is almost verbatim of the proof of Theorem \ref{main} and we sketch the proof.
%The following lemma helps to obtain the main result of this section.
 
 \begin{lemma}\label{bi-lemma}
 If $G$ is a vertex decomposable bipartite graph, then so is $\B_G$. 
\end{lemma}
%\textcolor{red}{Is it not the case that if $G$ is bipartite then $\mathcal B_G = G$? So if $G$ is vertex decomposasble bipartite, then $\mathcal B_G$ is by assumption?}

\begin{proof}
Let $V(G)=\{x_1,\ldots,x_n\} \cup \{y_1,\ldots,y_m\}$,
$(a_1,\ldots,a_l)$ be a shedding order of $G$ and $\{b_1,\ldots,b_s\}$
be an $i$-order of $G$.
Let $G_2$ be the graph as in Construction \ref{construction}.
\vskip 1mm \noindent
\textbf{Claim:} $G_2$ is a $W$-graph.
 \vskip 0.05mm
 \noindent
 \textit{Proof of the claim:}  
 We prove the claim by induction on
 $|V(G)|$.  
 If $|V(G)|=2$, then by \cite[Theorem 3.6]{selva1}, $G_2$ is a vertex decomposable
 bipartite graph. Therefore,
 $G_2 \setminus N_{G_2}[B]$ has a simplicial vertex for any 
 independent set $B$ of $G_2$. Assume that $|V(G)| \geq 3$. 
 Let $A$ be an independent set of 
 $G_2$. Suppose $A=\emptyset$. Since $G$ has a simplicial vertex, by 
 \cite[Lemma 3.1]{selva1}, $G_2$ has a simplicial vertex. Assume that 
 $A \neq \emptyset$.  Let $A=A_1 \coprod A_2 \coprod A_3$ where
 $A_1 \subseteq \{x_{i,1} \mid 1 \leq i \leq n\}$,
 $A_2 \subseteq \{x_{i,2}, y_{j,1} \mid 1 \leq i \leq n,~ 1 \leq j \leq m \}$ and 
 $A_3 \subseteq \{y_{i,2} \mid 1 \leq i \leq n\}$.
  If $A_1 \neq \emptyset$, then proceeding as in Theorem \ref{main} of the proof,
  one can show that, by induction on $|V(G)|$, 
  $G_2 \setminus N_{G_2}[A]$ has a simplicial vertex.
Assume that $A_1= \emptyset$. Note that 
\[
G_2 \setminus N_{G_2}[A_2 \coprod \{y_{1,2},\ldots,y_{m,2}\}] =L \setminus N_{L}[A_2]
\text{ where } L=G_2 \setminus \{x_{1,1},\ldots,x_{n,1}, y_{1,2},\ldots,y_{m,2}\}.
\]
Since $L \simeq G$ and $A_2$ is an independent set of 
$L$, $L \setminus N_{L}[A_2]$ has a simplicial vertex, say $z_{p,q}$.
If $z_{p,q} \in \{x_{1,2},\ldots,x_{n,2}\}$, then 
$z_{p,q}$ is a simplicial vertex of $G_2 \setminus N_{G_2}[A]$.
Suppose $z_{p,q} \in \{y_{1,1},\ldots,y_{m,1}\}$. Set $z_{p,q}=y_{j,1}$ for some 
$1 \leq j \leq m$. If $y_{j,2} \in A_3$, then $y_{j,1}$ is a simplicial vertex of 
$G_2 \setminus N_{G_2}[A]$.
If $y_{j,2} \notin A_3$, then $y_{j,2}$ is a simplicial vertex of 
$G_2 \setminus N_{G_2}[A]$.
Hence the claim.
%Therefore $G_2 \setminus N_{G_2}[A]$ has a simplicial vertex for any 
%independent set $A$ of $G_2$.

By Theorem \ref{russ-rs}, $G_2$ is vertex decomposable.
Since $\{b_{1,2},\ldots,b_{s,2}\}$ is an independent set of $G_2$,
 by Theorem \ref{ind-sub}, $G_2 \setminus N_{G_2}[b_{1,2},\ldots,b_{s,2}]$ is 
 a vertex decomposable graph. Note that 
 $G_2 \setminus N_{G_2}[b_{1,2},\ldots,b_{s,2}]=\B_G$. Hence
 $\B_G$ is a vertex decomposable graph.
 \end{proof}

Now the main theorem of this section can be derived from the above results.
 \begin{theorem}\label{main-bipartite}
 Let $G$ be a bipartite graph. Then the following are equivalent:
 \begin{enumerate}
  \item $G$ is a vertex decomposable graph;
  \item $J(G)^k$ has linear quotients for all $k \geq 1$;
  \item $J(G)^k$ is a componentwise linear ideal for all $k \geq 1$;
  \item $J(G)^k$ is a componentwise linear ideal for some $k >1$.
 \end{enumerate}

\end{theorem}
\begin{proof}
 $(1)\Rightarrow (2)$ 
 Since $G \setminus N_G[A]$
is a vertex decomposable bipartite graph for any independent set $A$ of $G$, 
by Lemma \ref{bi-lemma}, $\B_{G \setminus N_G[A]}$ is vertex decomposable.
By Theorem \ref{main}, \cite[Corollary 2.6]{GRV05}, 
$J(G)^{k}$  has linear quotients for all $k \geq 1$.
\vskip 0.5mm
\noindent
 $(2)\Rightarrow (3)$ This implication follows from Theorem \ref{lq-rs}. 
 \vskip 0.5mm
\noindent
 $(3)\Rightarrow (4)$ This implication is trivial.
 \vskip 0.5mm
\noindent
$(4)\Rightarrow (1)$ 
It follows from Theorem \ref{bi-not-cl} that 
$J(G)$ is componentwise linear. Then $G$ is a 
sequentially Cohen-Macaulay graph. Therefore, by Theorem \ref{adam-rs},
$G$ is a vertex decomposable graph.
\end{proof}

	Since the regularity of a componentwise linear ideal can be computed from its generators, 
we obtain a formula for the regularity of  powers of vertex cover ideals of
bipartite graphs
in terms of the maximum size of minimal vertex covers of graph.

\begin{corollary}
 If $G$ is a vertex decomposable bipartite graph, then, 
 \[
  \reg(J(G)^k)=k \deg(J(G)), \text{ for all }k \geq 1, 
 \]
where $\deg(J(G))=\max \{\deg(f) \mid f \text{ is a minimal generator of } J(G) \}$.
\end{corollary}

Since a tree is a bipartite graph, we derive, from Theorem \ref{main-bipartite},
one of the main results of Kumar and Kumar:

\begin{corollary}\cite[Corollary 3.4]{KK20}\label{tree-kk}
 If $G$ is a tree, then $J(G)^k$ has linear quotients, for all $k \geq 1$, and
 hence it is componentwise linear.
\end{corollary}

\begin{remark}
 We would like to note here that Corollary \ref{tree-kk} can be derived from Theorem \ref{main}.
 If $G$ is a forest, then all the spanning bipartite subgraphs are forests. It follows from
 \cite[Corollary 7]{Wood2} that $\B_{G \setminus N_G[A]}$ is vertex decomposable.
 Therefore, by Theorem \ref{main} and \cite[Corollary 2.6]{GRV05}, $J(G)^k$ is a 
 componentwise linear for all $k \geq 2$.
\end{remark}

Additionally, Theorem \ref{main-bipartite} allows us to recover 
Seyed Fakhari's result:

\begin{corollary}\cite[Corollary 3.7]{Fakhari}\label{fak-rs} Let $G$ be a bipartite graph.
 Then the following are equivalent:
 \begin{enumerate}
  \item $J(G)$ has linear resolution;
  \item $J(G)^k$ has linear resolution for all $k \geq 2$;
  \item $J(G)^k$ has linear resolution for some $k \geq 2$.
 \end{enumerate}
\end{corollary}

 We conclude the paper by raising the following question. This question is inspired by
 our main results, Theorem \ref{not-cl}, Theorem \ref{suff-cond1}, Theorem \ref{main}.
 \begin{question}
  Let $G$ be a vertex decomposable graph.
  \begin{enumerate}
   \item If $J(G)^{(2)}$ is not componentwise linear, is it true that $J(G)^{(k)}$ is not 
   componentwise linear for all $k \geq 3$?
   \item If  $\B_{G \setminus N_G[A]}$ is  vertex decomposable for any independent set $A$ of 
   $G$, is it true that $J(G)^{(k)}$ is a componentwise linear ideal for all $k \geq 2$?
  \end{enumerate}

 \end{question}

 Data sharing not applicable to this article as no datasets were generated or analysed during the current study.
\vskip 2mm
\noindent
\textbf{Acknowledgement:} 
The authors would like to thank Huy T\`ai H{\`a}  for valuable
discussions. The authors extensively used \textsc{Macaulay2}, \cite{M2}, and the 
packages \textsc{EdgeIdeals}, \cite{FHV_software},
\textsc{SimplicialDecomposability}, \cite{Cook}, \textsc{SymbolicPowers}, 
\cite{symbolic-package},
for testing their computations.
 The first author is 
supported by DST, Govt of India under the 
DST-INSPIRE [DST/Inspire/04/2019/001353] Faculty Scheme.

%\nocite*{}
%\bibliographystyle{alpha}  %% or 
%\bibliographystyle{plain}    %% ???
\bibliographystyle{abbrv}
\bibliography{refs_reg} 

\begin{thebibliography}{10}

\bibitem{BFH15}
J.~Biermann, C.~A. Francisco, H.~T. H\`a, and A.~Van~Tuyl.
\newblock Partial coloring, vertex decomposability, and sequentially
  {C}ohen-{M}acaulay simplicial complexes.
\newblock {\em J. Commut. Algebra}, 7(3):337--352, 2015.

\bibitem{EHHV20}
E.~Carlini, H.~T. H\`a, B.~Harbourne, and A.~Van~Tuyl.
\newblock {\em \text{Ideals of powers and powers of ideals:} Intersecting
  algebra, geometry, and combinatorics}, volume~27 of {\em Lecture Notes of the
  Unione Matematica Italiana}.
\newblock Springer International Publishing, 2020.

\bibitem{CCR16}
I.~D. Castrill\'{o}n, R.~Cruz, and E.~Reyes.
\newblock On well-covered, vertex decomposable and {C}ohen-{M}acaulay graphs.
\newblock {\em Electron. J. Combin.}, 23(2):Paper 2.39, 17, 2016.

\bibitem{Cook}
D.~Cook~II.
\newblock Simplicial decomposability.
\newblock {\em J. Softw. Algebra Geom.}, 2:20--23, 2010.

\bibitem{crupi}
M.~Crupi, G.~Rinaldo, and N.~Terai.
\newblock Cohen-{M}acaulay edge ideal whose height is half of the number of
  vertices.
\newblock {\em Nagoya Math. J.}, 201:117--131, 2011.

\bibitem{DDAGHN}
H.~Dao, A.~De~Stefani, E.~Grifo, C.~Huneke, and L.~N\'{u}\~{n}ez Betancourt.
\newblock Symbolic powers of ideals.
\newblock In {\em Singularities and foliations. geometry, topology and
  applications}, volume 222 of {\em Springer Proc. Math. Stat.}, pages
  387--432. Springer, Cham, 2018.

\bibitem{Dirac61}
G.~A. Dirac.
\newblock On rigid circuit graphs.
\newblock {\em Abh. Math. Sem. Univ. Hamburg}, 25:71--76, 1961.

\bibitem{symbolic-package}
B.~Drabkin, E.~Grifo, A.~Seceleanu, and B.~Stone.
\newblock Calculations involving symbolic powers.
\newblock {\em J. Softw. Algebra Geom.}, 9(1):71--80, 2019.

\bibitem{DHNT20}
L.~X. Dung, T.~T. Hien, H.~D. Nguyen, and T.~N. Trung.
\newblock Regularity and {K}oszul property of symbolic powers of monomial
  ideals.
\newblock {\em Math. Z.}, 298(3-4):1487--1522, 2021.

\bibitem{eagon}
J.~A. Eagon and V.~Reiner.
\newblock Resolutions of {S}tanley-{R}eisner rings and {A}lexander duality.
\newblock {\em J. Pure Appl. Algebra}, 130(3):265--275, 1998.

\bibitem{Nursel}
N.~Erey.
\newblock Powers of ideals associated to {$(C_4,2K_2)$}-free graphs.
\newblock {\em J. Pure Appl. Algebra}, 223(7):3071--3080, 2019.

\bibitem{FH}
C.~A. Francisco and H.~T. H{\`a}.
\newblock Whiskers and sequentially {C}ohen-{M}acaulay graphs.
\newblock {\em J. Combin. Theory Ser. A}, 115(2):304--316, 2008.

\bibitem{FHV_software}
C.~A. Francisco, A.~Hoefel, and A.~Van~Tuyl.
\newblock Edge{I}deals: a package for (hyper)graphs.
\newblock {\em J. Softw. Algebra Geom.}, 1:1--4, 2009.

\bibitem{fv}
C.~A. Francisco and A.~Van~Tuyl.
\newblock Sequentially {C}ohen-{M}acaulay edge ideals.
\newblock {\em Proc. Amer. Math. Soc.}, 135(8):2327--2337, 2007.

\bibitem{GRV05}
I.~Gitler, E.~Reyes, and R.~H. Villarreal.
\newblock Blowup algebras of ideals of vertex covers of bipartite graphs.
\newblock In {\em Algebraic structures and their representations}, volume 376
  of {\em Contemp. Math.}, pages 273--279. Amer. Math. Soc., Providence, RI,
  2005.

\bibitem{M2}
D.~R. Grayson and M.~E. Stillman.
\newblock Macaulay2, a software system for research in algebraic geometry.
\newblock Available at \url{http://www.math.uiuc.edu/Macaulay2/}.

\bibitem{GTS20}
Y.~Gu, H.~T. H\`a, and J.~W. Skelton.
\newblock Symbolic powers of cover ideals of graphs and {K}oszul property.
\newblock {\em Internat. J. Algebra Comput.}, 31(5):865--881, 2021.

\bibitem{HV22}
H.~T. H\`a and A.~Van~Tuyl.
\newblock Powers of componentwise linear ideals: the {H}erzog-{H}ibi-{O}hsugi
  conjecture and related problems.
\newblock {\em Res. Math. Sci. (to appear)}.

\bibitem{HerHibi}
J.~Herzog and T.~Hibi.
\newblock Componentwise linear ideals.
\newblock {\em Nagoya Math. J.}, 153:141--153, 1999.

\bibitem{Herzog'sBook}
J.~Herzog and T.~Hibi.
\newblock {\em Monomial ideals}, volume 260 of {\em Graduate Texts in
  Mathematics}.
\newblock Springer-Verlag London, Ltd., London, 2011.

\bibitem{HHM20}
J.~{Herzog}, T.~{Hibi}, and S.~{Moradi}.
\newblock {Componentwise linear powers and the $x$-condition}.
\newblock {\em arXiv e-prints}, Oct. 2020.

\bibitem{HerHibiOhsugi}
J.~Herzog, T.~Hibi, and H.~Ohsugi.
\newblock Powers of componentwise linear ideals.
\newblock In {\em Combinatorial aspects of commutative algebra and algebraic
  geometry}, volume~6 of {\em Abel Symp.}, pages 49--60. Springer, Berlin,
  2011.

\bibitem{HerReiWelker}
J.~Herzog, V.~Reiner, and V.~Welker.
\newblock Componentwise linear ideals and {G}olod rings.
\newblock {\em Michigan Math. J.}, 46(2):211--223, 1999.

\bibitem{HerTak}
J.~Herzog and Y.~Takayama.
\newblock Resolutions by mapping cones.
\newblock {\em Homology Homotopy Appl.}, 4(2, part 2):277--294, 2002.
\newblock The Roos Festschrift volume, 2.

\bibitem{JahanZheng}
A.~S. Jahan and X.~Zheng.
\newblock Ideals with linear quotients.
\newblock {\em J. Combin. Theory Ser. A}, 117(1):104--110, 2010.

\bibitem{jayanthan}
A.~V. Jayanthan, N.~Narayanan, and S.~Selvaraja.
\newblock Regularity of powers of bipartite graphs.
\newblock {\em J. Algebraic Combin.}, 47(1):17--38, 2018.

\bibitem{JS21}
A.~V. Jayanthan and S.~Selvaraja.
\newblock Upper bounds for the regularity of powers of edge ideals of graphs.
\newblock {\em J. Algebra}, 574:184--205, 2021.

\bibitem{KK20}
A.~Kumar and R.~Kumar.
\newblock On the powers of vertex cover ideals.
\newblock {\em J. Pure Appl. Algebra}, 226(1):Paper No. 106808, 10, 2022.

\bibitem{Fatemesh11}
F.~Mohammadi.
\newblock Powers of the vertex cover ideal of a chordal graph.
\newblock {\em Comm. Algebra}, 39(10):3753--3764, 2011.

\bibitem{Fatemesh14}
F.~Mohammadi.
\newblock Powers of the vertex cover ideals.
\newblock {\em Collect. Math.}, 65(2):169--181, 2014.

\bibitem{FS10}
F.~Mohammadi and S.~Moradi.
\newblock Weakly polymatroidal ideals with applications to vertex cover ideals.
\newblock {\em Osaka J. Math.}, 47(3):627--636, 2010.

\bibitem{NPY21}
N.~Nemati, M.~R. Pournaki, and S.~Yassemi.
\newblock Componentwise linearity and the gcd condition are preserved by the
  polarization.
\newblock {\em Bull. Math. Soc. Sci. Math. Roumanie (N.S.)},
  64(112)(4):391--399, 2021.

\bibitem{selva1}
S.~Selvaraja.
\newblock Symbolic powers of vertex cover ideals.
\newblock {\em Internat. J. Algebra Comput.}, 30(6):1167--1183, 2020.

\bibitem{Fakhari}
S.~A. Seyed~Fakhari.
\newblock Symbolic powers of cover ideal of very well-covered and bipartite
  graphs.
\newblock {\em Proc. Amer. Math. Soc.}, 146(1):97--110, 2018.

\bibitem{Se20}
S.~A. Seyed~Fakhari.
\newblock On the minimal free resolution of symbolic powers of cover ideals of
  graphs.
\newblock {\em Proc. Amer. Math. Soc.}, 149(9):3687--3698, 2021.

\bibitem{adam}
A.~Van~Tuyl.
\newblock Sequentially {C}ohen-{M}acaulay bipartite graphs: vertex
  decomposability and regularity.
\newblock {\em Arch. Math. (Basel)}, 93(5):451--459, 2009.

\bibitem{VanVilla}
A.~Van~Tuyl and R.~H. Villarreal.
\newblock Shellable graphs and sequentially {C}ohen-{M}acaulay bipartite
  graphs.
\newblock {\em J. Combin. Theory Ser. A}, 115(5):799--814, 2008.

\bibitem{west}
D.~B. West.
\newblock {\em Introduction to graph theory}.
\newblock Prentice Hall, Inc., Upper Saddle River, NJ, 1996.

\bibitem{Wood2}
R.~Woodroofe.
\newblock Vertex decomposable graphs and obstructions to shellability.
\newblock {\em Proc. Amer. Math. Soc.}, 137(10):3235--3246, 2009.

\bibitem{Russ11}
R.~Woodroofe.
\newblock Chordal and sequentially {C}ohen-{M}acaulay clutters.
\newblock {\em Electron. J. Combin.}, 18(1):Paper 208, 20, 2011.

\end{thebibliography}
\end{document}